\documentclass[preprint,12pt,authoryear]{elsarticle}

\usepackage[utf8]{inputenc}

\usepackage{amsmath}
\usepackage{amssymb}

\usepackage{amsthm}

\usepackage{url}
\usepackage{algorithm}
\usepackage{algorithmic}
\usepackage{booktabs}       

\usepackage{subcaption}     
\usepackage{placeins}
\usepackage{wrapfig}
\usepackage{xcolor}
\usepackage{comment}

\usepackage{thmtools}

\newcommand{\T}{\mathsf{T}}
\DeclareMathOperator*{\Tr}{\mathsf{Tr}}

\journal{}

\begin{document}

\begin{frontmatter}



\title{Score-based Transport Modeling for Mean-Field Fokker-Planck Equations}

\author{Jianfeng LU \fnref{D}}
\ead{jianfeng@math.duke.edu}

\fntext[D]{Departments of Mathematics, Physics, and Chemistry, Duke Univeristy, Durham, USA}

\author{Yue WU \fnref{H}}
\ead{ywudb@connect.ust.hk}
\fntext[H]{Department of Mathematics, The Hong Kong University of Science and Technology, Hong Kong, China}

\author{Yang XIANG \fnref{H,S}}
\ead{maxiang@ust.hk}
\fntext[S]{HKUST Shenzhen-Hong Kong Collaborative Innovation Research Institute, Futian, Shenzhen, China}


\begin{abstract}
We use the score-based transport modeling method
to solve the mean-field Fokker-Planck equations, 
which we call MSBTM.
We establish an upper bound on the time derivative of the Kullback-Leibler (KL) divergence to MSBTM numerical estimation from the exact solution, thus validates the MSBTM approach. 
Besides, we provide an error analysis for the algorithm.
In numerical experiments, we study two types of mean-field Fokker-Planck equation and their corresponding dynamics of particles in interacting systems.
The MSBTM algorithm is numerically validated through qualitative and quantitative comparison between the MSBTM solutions, the results of integrating the associated stochastic differential equation and the analytical solutions if available.
\end{abstract}




\end{frontmatter}


\section{Introduction}
\label{sec:introduction}
The Fokker-Planck equation \citep{risken1996fokker} describes the time evolution of the probability density function of the observables such as particle systems in physics and statistical mechanics \citep{chandler1987introduction, spohn2012large}.
In this paper, we are interested in the mean-field Fokker-Planck equation \citep{Huang2020, Guillin2021, Bresch2022} 
of the time-dependent probability density function $\rho_t(x): \Omega \to \mathbb{R},$
\begin{equation}
\label{eq:FPE}
\tag{MFPE}
	\begin{aligned}
		\frac{\partial}{\partial t} \rho_t(x) &= -\nabla \cdot (b_t(x, \rho_t) \rho_t(x) -D_t(x) \nabla \rho_t(x)), \\
		b_t(x, \rho_t) &= f_t(x) - \int_{\Omega} K(x,y) \rho_t(y) dy,
		\quad x \in \Omega \subseteq \mathbb{R}^{d}.
	\end{aligned}
\end{equation}
For clarity, the subscript $t$ of a function indicates the time dependence of the function.
Here the drift vector $b_t(x, \rho_t)   \in \mathbb{R}^{d}$ depends on the density function $\rho_t(x)$ through the mean field interaction term $ \int_{\Omega} K(x,y) \rho_t(y) dy$
with the interaction function  
$K(x,y):\mathbb{R}^{d} \times \mathbb{R}^{d} \to \mathbb{R}^{d}$,
for example $K(x,y) = x-y$.
The diffusion matrix $D_t(x)  \in \mathbb{R}^{d \times d}$ is positive definite.
%
%
Consider a system of $N$ particles in $\mathbb{R}^{d}$ moving according to the stochastic dynamic defined by
\begin{equation}\label{eq:NP sde}
	dx^{(i)}_t = \bigl( f_t(x^{(i)}_t)  - \frac{1}{N} \sum_{1\leq j \leq N} K(x^{(i)}_t, x^{(j)}_t) \bigr) dt
	+\nabla \cdot D_t (x^{(i)}_t) dt + \sqrt{2} \sigma_t(x^{(i)}_t) dW^{i}_t,
\end{equation}
for $i=1,\dots N,$
and we have $x_t^{(i)}$
each particle position at time $t\geq 0$ ,
$D_t(x) = \sigma_t(x) \sigma_t(x)^\T, \sigma_t(x) \in \mathbb{R}^{d\times d}$, and
$\left\{ W^{i}_t \right\}_{ t\geq 0 }^{i=1,\dots,N} $ are independent standard Brownian motions in $\mathbb{R}^{d}.$
When the number of particles $N$ goes to infinity, from the theory of the propagation of chaos,
each particle at time $t$ asymptotically admits the identical probability density function $\rho_t(x),$
the pairwise interaction $\frac{1}{N} \sum_{1\leq j \leq N} K(x^{(i)}_t, x^{(j)}_t) $ becomes the interaction $\int_{\Omega} K(x^{(i)}_t,y) \rho_t(y) dy$ between any particle $x^{(i)}_t$ and its distribution,
and the density function $\rho_t(x)$ satisfies the mean-field Fokker-Planck equation~\eqref{eq:FPE}, see e.g., \citep{jabin2014review, jabin2017mean}.

Standard numerical schemes for solving partial differential equations, including the Fokker-Planck equation~\eqref{eq:FPE}, 
directly discretize the time and space, then construct and solve the discrete equations.
The computational complexity of standard numerical methods exponentially grows with the dimension $d$.
On the other hand, numerical integration of the associated stochastic differential equation (SDE)~\eqref{eq:NP sde} has no direct access to the value of the density and some other quantities that require the density value such as the entropy, the probability current.
To scale well in high dimensional problem and to obtain the density value at some position, a recent work \citep{Boffi} proposed the score-based transport modeling (SBTM) method to solve the Fokker-Planck equation.

In this paper, we extend SBTM to solving the Fokker-Planck equation with mean field interaction, 
which is a more general case compared to the problem in \citep{Boffi}.
We name the SBTM that solves the mean-field Fokker-Planck equations the mean field SBTM (MSBTM).
Furthermore,
the presence of mean field interaction brings the advantage of saving computational complexity.
Specifically,
when calculating the probability density function of $N$ particles $x^{(i)} \in \mathbb{R}^d, i =1, \dots,N$ in an interacting system, 
\citep{Boffi} reformulates the pairwise interaction as a function of concatenated variable $(x^{(1)}, \dots, x^{(N)}) \in \mathbb{R}^{Nd}$,
as result the Fokker-Planck equation becomes a problem in $\mathbb{R}^{Nd}.$
Moreover, another neural network is needed to impose the permutation symmetry of $N$ particles.
When $N$ is large, the dimensionality of the Fokker-Planck equation $Nd$ will be very high even for small $d$.
In contrast, from the mean field perspective, each particle obeys the identical distribution and 
automatically the particles are permutation symmetric,
and the associated mean-field Fokker-Planck equation remains in $\mathbb{R}^{d}$.
Therefore, if there are many particles in the interacting system ($N$ is sufficiently large), 
the MSBTM method avoids scaling up dimensions by the number of particles and thus significantly saves the computation.
The method of MSBTM to solve the mean-field Fokker-Planck equation is validated by numerical examples in section~\ref{sec:experiments}.

\FloatBarrier
\subsection*{Related Works}

\paragraph{Score-based diffusion modeling} 
Recent works \citep{yangsong2019, yangsong2021} in generative modeling use deep neural network $s (\cdot) $ 
to approximate the  score function $\nabla \log \rho $ of a target density $\rho$, and propagate the samples using SDE $dX_{\eta} = s(X_{\eta}) + dW_\eta $ based on the estimated score, as a result, the generated samples obey the target distribution. This approach is known as the score-based diffusion modeling. It has achieved state-of-the-art performance in generative modeling of images. 

\paragraph{Score-based transport modeling (SBTM)}
Inspired by score-based diffusion modeling, SBTM method \citep{Boffi} trains time-dependent deep neural networks to approximate the score $\nabla \log \rho_t$ 
and solves the Fokker-Planck equation defined by
\begin{equation}\label{eq:FPEsbtm}
    \frac{\partial}{\partial t} \rho_t(x) = -\nabla \cdot(b_t(x) \rho_t(x)- D_t(x) \nabla \rho_t(x)),
    \quad x \in \mathbb{R}^{d}.
\end{equation}
To be specific,  SBTM method reformulates the Fokker-Planck equation~\eqref{eq:FPEsbtm} as a transport equation 
\begin{equation}\label{eq:TEsbtm}
    \frac{\partial}{\partial t} \rho_t(x) = - \nabla \cdot(v_t(x) \rho_t(x)),
\end{equation}
 with the velocity field $v_t(x) = b_t(x)- D_t(x) \nabla \log \rho_t(x).$
Given the velocity field, the probability flow equation is defined by
\begin{equation}
	\label{eq:probflowsbtm}
   \frac{d}{dt} X_{\tau,t}(x) = v_t(X_{\tau,t}(x)), \qquad X_{\tau,\tau}(x) = x, \quad t,\tau \ge 0.
\end{equation}
The transport map $X_{\tau, t}(x): \mathbb{R}^{d} \to \mathbb{R}^{d}$ represents the position at time $t$ when occupying position $x$ at time $\tau$ and propagating according to the equation~\eqref{eq:probflowsbtm}.
In particular, we have $X_{t',t}( X_{t,t'}(x) ) = x $ and $X_{t',\tau}(X_{t,t'}(x) ) =X_{t,\tau}(x).$
In equation~\eqref{eq:probflowsbtm}, let $\tau=0$ and $x$ be a sample from $\rho_0$, then $X_{0,t}(x)$ will be a sample from $\rho_t$ that solves equation~\eqref{eq:TEsbtm}, or equivalently the Fokker-Planck equation~\eqref{eq:FPEsbtm}.
It can be shown that
\begin{equation}
	\label{eq:rhot:rho0 pre}
	\rho_t(X_{0,t}(x)) = \rho_0(X_{0,0}(x)) \exp\left( - \int_0^t \nabla \cdot v_{t'} (X_{0,t'}(x)) dt' \right).
\end{equation}
To evaluate the value of $\rho_t$ at $x \in \Omega,$
substitute  $x=X_{t,0}(x)$ into equation~\eqref{eq:rhot:rho0 pre} and then
\begin{equation}
	\label{eq:rhot:rho0}
	\rho_t(x) = \rho_0(X_{t,0}(x)) \exp\left( - \int_0^t \nabla \cdot v_{t'} (X_{t,t'}(x)) dt' \right).
\end{equation}
Here $ X_{t,0}(x)$ is obtained by solving equation~\eqref{eq:probflowsbtm} backward from $X_{t,t}(x)=x$ given $x.$
We can regard $\rho_t $ as the pushforward of $\rho_0$ under the flow map $X_{0, t}(\cdot)$
that satisfies the ordinary differential equation~\eqref{eq:probflowsbtm}.
To summarize, SBTM trains the neural networks $s_t(x)$ to approximate the score 
$\nabla \log \rho_t(x)$ and then estimates the velocity $v_t(x)$ using $ b_t(x)- D_t(x) s_t(x)$, as a result SBTM can estimate pointwise density value $\rho_t(x).$
Compared to the score-based diffusion model that trains the neural networks on the data from the target distribution,
the neural networks in SBTM method are trained over samples generated by the probability flow equation~\eqref{eq:probflowsbtm}, without the knowledge of the solution density.

In this paper, we follow the transport modeling and study a more general case where the drift term $b_t(x, \rho_t)=f_t(x) - \int_\Omega K(x,y)\rho_t(y) dy $ involves a mean field interaction term and thus has dependence on the density.

\section{Methodology}
\label{sec:methodology}
In this section,
we introduce MSBTM that extends the score-based transport modeling to solving the mean-field Fokker-Planck equation as~\eqref{eq:FPE}:
\begin{equation}
\tag{MFPE}
	\begin{aligned}
		\frac{\partial}{\partial t} \rho_t(x) &= -\nabla \cdot (b_t(x, \rho_t) \rho_t(x) -D_t(x) \nabla \rho_t(x)), \\
		b_t(x, \rho_t) &= f_t(x) - \int_{\Omega} K(x,y) \rho_t(y) dy,
		\quad x \in \Omega \subseteq \mathbb{R}^{d}.
	\end{aligned}
\end{equation}

\subsection{Preliminaries and assumptions}
\label{sec:assumptions}
Let $s_t: \mathbb{R}^{d} \to \mathbb{R}^{d}$ denote the neural network approximating the score $\nabla\log \rho_t$ at time $t$.
We consider the transport equation defined by 
\begin{equation}
	\tag{TE}
	\label{eq:transport}
	\begin{aligned}
		\frac{\partial}{\partial t} \rho_t(x) &= - \nabla \cdot \left( v_t(x, \rho_t) \rho_t(x)\right) ,   \\
		v_t(x, \rho_t ) &= b_t(x, \rho_t) - D_t(x) s_t(x),  \\
		b_t(x, \rho_t) &= f_t(x) - \int_{\Omega} K(x,y) \rho_t(y) dy,
	\end{aligned}
\end{equation}
and denote its solution $\rho_t: \Omega \to \mathbb{R}_{>0} $.
Note that 
when the neural network $s_t$ exactly learns the score $\nabla \log \rho_t$ over the training samples, 
then the transport equation~\eqref{eq:transport} is equivalent to the mean-field Fokker-Planck equation~\eqref{eq:FPE}, 
and $\rho_t$ coincides with $\rho_t^{\star}$, the exact solution of \eqref{eq:FPE}.

Denote $\|\cdot \|$ the $L_2$ norm for a vector or the induced norm for a matrix.
%
We make a few assumptions: 

(A.1) 
We assume that the stochastic process~\eqref{eq:NP sde} associated with the Fokker-Planck equation~\eqref{eq:FPE} evolves over a domain $\Omega\subseteq \mathbb{R}^d$ at all times~$t\ge0$,
the vector-valued function $f_t: \Omega \times \mathcal{P}(\Omega) \to \mathbb{R}^d,$ 
and the diffusion matrix $D_t: \Omega \to \mathbb{R}^{d\times d}$ are twice-differentiable and bounded in both $x$ and $t$, 
the interaction function $K(x,y): \Omega \times \Omega \to \mathbb{R}^d$ is twice-differentiable and bounded,
so that the solution to the SDE~\eqref{eq:NP sde} is well-defined at all times $t\ge 0$.
The symmetric matrix $ D_t(x) = D_t^\T(x) $ is assumed
to be positive definite for each $(t, x)$.

(A.2)
We assume that the initial probability density function $\rho_0$ is twice-differentiable,
positive everywhere on $\Omega$, and such that 
\begin{equation*}
    E_0 = -\int_\Omega \rho_0(x) \log \rho_0(x) dx < \infty.
\end{equation*}
This guarantees that $\rho^{\star}_t$ enjoys the same properties at all times $t>0$.
In particular we assume $\rho_t$ and $\nabla \log \rho_t$ are differentiable on $\Omega$.  

(A.3) Finally, we assume that 
$\log \rho^{\star}_t$ is $L $-smooth globally in $(t,x)\in [0,\infty)\times \Omega$, i.e. there exists a constant $L > 0$ satisfying
\begin{equation}
	\label{eq:Ksmooth}
	\quad \forall (t,x)\in [0,\infty)\times \Omega , \quad 
 \| \nabla \log \rho^{\star}_t(x) - \nabla \log \rho^{\star}_t(y) \| \le L \|x-y \|,
\end{equation}
so that the solution of the probability flow equation exists and is unique.

\subsection{Score-based transport modeling for mean-field Fokker-Planck equations}
\label{sec:main results}

\smallskip 

\begin{restatable}[]{proposition}{sbtmseq}
\label{prop:msbtmseq}
Given the score approximation function $s_t(x)$, 
define $v_t(x, \rho_t)= b_t(x, \rho_t) -D_t(x)s_t(x)$
and let $X_{0,t}$ solve the equation 
	\begin{equation}
		\label{eq:X}
			\frac{d}{dt}X_{0,t}(x) = v_t(X_{0,t}(x), \rho_t(X_{0,t}(x))),
			\quad X_{0,0}(x)  = x.
	\end{equation}    	
Fix $t\ge 0$ and consider the optimization problem
	\begin{equation}
		\label{eq:sbtm3}
		\tag{MSBTM}
		\begin{aligned}
			\min_{s_t} \int_\Omega \left( \|s_t(X_{0,t}(x)) \|^2 +2 \nabla \cdot s_t(X_{0,t}(x)) \right) \rho_0(x) dx.
		\end{aligned}
	\end{equation}
Under the assumptions in section~\ref{sec:assumptions},
the minimizer~$\hat{s}_t$ of~\eqref{eq:sbtm3} satisfies
	 $\hat{s}_t(x) = \nabla \log \rho_t^{\star}(x)$  where $\rho^{\star}_t $ solves the mean-field Fokker-Planck equation~\eqref{eq:FPE},
then the solution $\rho_t$ of the transport equation associated to this minimizer $\hat{s}_t$
coincides with $\rho^{\star}_t$.
Moreover, we have 
 \begin{equation}
	 	\label{eq:entD}
	 	\frac{d}{dt} D_{\mathrm{KL}}(\rho_t \Vert \rho^{\star}_t) \le \frac12 
	 	\max \left\{ \hat{D}_t, \hat{C}_{K,t}\right\} 
	 	\int_\Omega \|s_t(X_{0,t}(x)) -  \nabla \log \rho_t(X_{0,t}(x)) \|^2 \rho_0(x)dx,
	 \end{equation}
where constants $\hat{D}_t, \hat{C}_{K,t}>0 $ are defined in Lemma~\ref{lemma:entropy}.
\end{restatable}

This proposition shows that the minimizer of the optimization problem~\eqref{eq:sbtm3} is the score of the solution to the mean-field Fokker-Planck equation~\eqref{eq:FPE}.
In addition, the time derivative of the KL divergence to the transport equation solution $\rho_t $ from the target solution $\rho_t^{\star}$
is upper bounded by the error between the neural network $s_t$ and the score $\nabla \log \rho_t$, thus is
controlled by the objective function in \eqref{eq:sbtm3}.
Provided the inequality~\eqref{eq:entD}, if we use the same initial density $\rho_0= \rho_0^{\star}$, the total KL divergence can be controlled: 
\begin{equation}
\begin{aligned}
    & D_{\mathrm{KL}}(\rho_T \Vert \rho^{\star}_T)
     \le \frac12 \int_0^T 
		\max \left\{ \hat{D}_t, \hat{C}_{K,t}\right\} 
		\int_{\Omega} \|s_t(X_{0,t}(x)) - \nabla \log \rho_t(X_{0,t}(x)) \|^2 \rho_0(x) dx dt\\
    & \lessapprox  \frac12 \sum_{t_k} \Delta t
		\max \left\{ \hat{D}_{t_k}, \hat{C}_{K,t_k}\right\}  \int_{\Omega} \|s_t(X_{0, t_k}(x)) - \nabla \log \rho_t(X_{0,t_k}(x)) \|^2 \rho_0(x) dx.
\end{aligned}
\end{equation}
Thus sequentially minimizing the objective in \eqref{eq:sbtm3} can accurately approximate solution to the mean-field Fokker-Planck equation~\eqref{eq:FPE}.

The proof of Proposition~\ref{prop:msbtmseq} can be found in \ref{app:proof p1}.

\subsection{MSBTM algorithm}
\label{sec:algorithm}

Here we summarize Algorithm~\ref{alg:msbtm} of the MSBTM method that solves the mean-field Fokker-Planck equation~\eqref{eq:FPE} and provide an error analysis.
Given $N$ samples from the initial density $\rho_0,$
we train the neural network $s_{t}$ to minimize the objective in equation~\eqref{eq:sbtm3} over the samples and then propagate the samples according to equation~\eqref{eq:X} alternatively.
For the velocity field in equation~\eqref{eq:X},
we approximate the mean field interaction term by the empirical mean of $N$ samples $\left\{x^{(i)} \right\}_{i=1}^{N}\sim \rho_0,$
\begin{equation}
    \begin{aligned}
\int_\Omega K( x^{(i)}, y)\rho_t(y) dy  & =
    \int_\Omega K( X_{0,t}(x^{(i)}), X_{0,t}(y) ) \rho_0(y) dy  \\
    & \approx \frac{1}{N} \sum_{j=1}^N K( X_{0,t}(x^{(i)}),  X_{0,t}(x^{(j)})),
    \end{aligned}
\end{equation}
with $\left\{ X_{0,t}(x^{(i)}) \right\}_{i=1}^N \sim \rho_t.$ 
We denote $X_{t_0,t_k}(x^{(i)}) $ as $X_{t_k}^{(i)}$ for short in Algorithm~\ref{alg:msbtm}.

\begin{algorithm}[!hptb]
	\caption{Mean field score-based transport modeling.}
	\label{alg:msbtm}
	\begin{algorithmic}
		\STATE \textbf{Input}: An initial time $t_0 \geq 0$. A set of $N $ samples 
  $\{x^{(i)}\}_{i=1}^N$ from initial density $\rho_{t_0}$. A time step $\Delta t$ and the number of steps $N_{T}.$
		\STATE Initialize sample locations $X_{t_0}^{(i)} = x^{(i)}$ for $i = 1, \dots, N$.
		\FOR{$k=0, \dots, N_T $}
		\STATE Optimize: $s_{t_k} = \arg\min_{s} \frac{1}{N} \sum_{i=1}^N \left[ \|s(X_{t_k}^{(i)} ) \|^2 + 2 \nabla\cdot s(X_{t_k}^{(i)})\right]$. 
		\STATE Propagate the samples for $i = 1, \dots, N$: 
		$$ \quad X_{t_{k+1}}^{(i)} = X_{t_{k}}^{(i)}
         + \Delta t \left[  f_t(X_{t_k}^{(i)}) 
         - \frac1N \sum_{j=1}^N K(X_{t_k}^{(i)}, X_{t_k}^{(j)} )  
        -D_{t_{k}}(X_{t_k}^{(i)}) s_{t_{k}}(X_{t_k}^{(i)}) \right]. $$
		\STATE Set $t_{k+1} = t_{k} + \Delta t$.
		\ENDFOR
		\STATE \textbf{Output}: $N$ samples $\{X_{t_k}^{(i)} \}_{i=1}^N$ from $\rho_{t_k}$ and the score $\{s_{t_k} (X_{t_k}^{(i)}) \}_{i=1}^N$  for all $\{t_k\}_{k=0}^{N_T}$. 
	\end{algorithmic}
\end{algorithm}


\begin{restatable}[Error analysis]{proposition}{error}
\label{prop:error}
Let $N, N_T, \Delta t, t_0, \rho_{t_0}$ be defined as in Algorithm~\ref{alg:msbtm}
and denote $\epsilon:= \sup_{x \in \Omega , t_0 \leq t \leq t_0 + N_T \Delta t} \Vert s_t(x) - \nabla \log \rho^{\star}_t(x) \Vert $ as the approximation error of well-trained neural networks $s_t$,
with $\rho_t^{\star}$ the solution of equation~\eqref{eq:FPE}.
Denote $X_{t_0, t}^{N }(\cdot)$ the numerical transport mapping 
obtained from the MSBTM algorithm.
Let $X^{\star}_{t_0, t}(\cdot)$ solve the probability flow equation
\begin{equation}
    \frac{d}{dt} X^{\star}_{t_0,t}(x) = v_t^{\star} (X^{\star}_{t_0, t}(x), \rho^{\star}(X^{\star}_{t_0, t}(x))), \quad X_{t_0, t_0}^{\star}(x) =x,
\end{equation}
with $v_t^{\star} (x, \rho_t^{\star})=f_t (x) - \int_{\Omega} K(x,y) \rho^{\star}_t(y) dy - D_t (x) \nabla \log \rho_t^{\star} (x),$
corresponding to equation~\eqref{eq:FPE}.
Under the assumption~\ref{sec:assumptions},
we have, for any $t \in \left\{ t_0 + n\Delta t \right\}_{n=0}^{N_T},$
\begin{equation}\label{eq:error-X}
    \mathbb{E}_{x\sim \rho_{t_0}}\left(\Vert X^{\star}_{t_0, t}(x) -X_{t_0, t}^{N }(x) \Vert \right)=
    (t-t_0 ) 
    \left[ \mathcal{O}(\frac{1}{\sqrt{N} }) + \mathcal{O}(\epsilon) + \mathcal{O}(\Delta t) \right]  ,
\end{equation}
as $N \rightarrow +\infty, \epsilon, \Delta t \rightarrow 0$.
\end{restatable}
\paragraph{Remark}
Here the errors come from three sources:
$\mathcal{O}(\frac{1}{\sqrt{N} })$
caused by Monte-Carlo method,
neural network approximation error 
$\mathcal{O}(\epsilon) $ and
discretization error $\mathcal{O}(\Delta t)$.
Note that the error $\epsilon$ can be small and does not accumulate with time since the time-dependent neural networks $s_t$ are trained sequentially in Algorithm~\ref{alg:msbtm}.
This error analysis indicates that we can take larger $N$ (more samples), smaller time step $\Delta t$ and smaller approximation error $\epsilon$ (neural network and training details) to achieve smaller numerical error.

The proof of Proposition~\ref{prop:error} is given in \ref{app:proof p2}.

\FloatBarrier
\section{Experiments}
\label{sec:experiments}

\subsection{Harmonically interacting particles in a harmonic trap}
\label{sec:ex1}

In the first example, we study the mean-field Fokker-Planck equation defined by
\begin{equation}
	\label{eq:harmonic_FPE}
	\frac{\partial}{\partial t} \rho_t(x) = -\nabla \cdot 
	\left[ 	\left( \beta_t - x  -  \alpha \int_{\Omega} (x -y) \rho_t(y) dy \right) \rho_t(x) -D \nabla \rho_t(x)
	\right],
\end{equation}
with $\rho_t(x): \mathbb{R}^{\bar{d}} \rightarrow \mathbb{R},  x \in \Omega\subseteq \mathbb{R}^{\bar{d}}$ and parameters $\alpha \in (0,1),  \beta_t \in \mathbb{R}^{\bar{d}}, D>0.$
If we consider a system of $N$ particles moving according to the stochastic dynamics
\begin{equation}
	\label{eq:harmonic_SDE1}
	dx^{(i)}_t = (\beta_t - x^{(i)}_t)dt  - 
	\frac{\alpha}{N} \sum_{j=1}^N 
	\left(x^{(i)}_t - x^{(j)}_t \right)dt
	+ \sqrt{2D}\,dW_t,\quad i=1, \dots, N,
\end{equation}
when $N$ goes to infinity, each particle becomes indistinguishable and asymptotically independent from the others, and its interaction becomes with its own distribution.
And each particle admits the density $\rho_t(x)$  that solves the equation~\eqref{eq:harmonic_FPE}.

Rewriting the dynamics~\eqref{eq:harmonic_SDE1} as follows
\begin{equation}
	\label{eq:harmonic_SDE}
	dx^{(i)}_t = (\beta_t - x^{(i)}_t)dt  - 
	\alpha \left( x^{(i)}_t - \frac{1}{N}\sum_{j=1}^N x^{(j)}_t \right)dt
	 + \sqrt{2D}\,dW_t,\quad i=1, \dots, N,
\end{equation}
which is similar to the example in \citep{Boffi}.
A physical interpretation of this system is that, $N$ particles repel each other according to a harmonic interaction with $\alpha$ setting the magnitude of the repulsion,
while experiencing harmonic attraction towards a moving trap $\beta_t$.
The dynamics~\eqref{eq:harmonic_SDE} is an Ornstein-Uhlenbeck process in the extended variable $x=(x^{(1)}, x^{(2)}, \dots, x^{(N)})^{\T} \in \mathbb{R}^{\bar{d}N}$ with block components $x^{(i)} \in \mathbb{R}^{\bar{d}}$. 
So this dynamics admits a tractable analytical solution in $\bar{d}N$-dimensional space:
assuming a Gaussian initial condition, the solution  $\rho_{\bar{d}N}: \mathbb{R}^{\bar{d}N} \rightarrow \mathbb{R}$ to the Fokker-Planck equation associated with~\eqref{eq:harmonic_SDE} is a  Gaussian for all time.
The joint distribution of concatenated variable $(x^{(1)}, x^{(2)}, \dots, x^{(N)})^{\T} $ is a Gaussian in $\mathbb{R}^{\bar{d}N},$ and converges to a product distribution as $N$ increases. 
Then the distribution of each component should also be a Gaussian.
Therefore, with a sufficiently large $N,$ the solution $x^{(i)}, i=1, \dots, N$ to~\eqref{eq:harmonic_SDE1}
approximately obey the identical Gaussian distribution
that solves the Fokker-Planck equation~\eqref{eq:harmonic_FPE}.
As a result, we can characterize the solution to~\eqref{eq:harmonic_FPE} by its mean $m_t \in \mathbb{R}^{\bar{d}}$ and covariance $C_t \in \mathbb{R}^{\bar{d} \times \bar{d}} $.

\paragraph{Setting}  In the experiment, we take $\bar{d} = 2, \beta_t = a(\cos\pi\omega t, \sin\pi\omega t)^\T$ with $a = 2$, $\omega = 1$, $D = 0.25$, $\alpha = 0.5$, and the number of samples $N=300$.
The particles are initialized from an isotropic Gaussian $\rho_0 (x) $ with mean $\beta_0$ (the initial trap position) and variance $\sigma_0^2 = 0.25$.
Different to the neural network representation in \citep{Boffi} that requires another one neural network to achieve symmetry in terms of particle permutation,
we directly parameterize the score $s_t:  \mathbb{R}^{2}\rightarrow \mathbb{R}^{2}$ using a two-hidden-layer neural network with $\mathtt{n\_hidden} = 32$ neurons per hidden layer, and $\texttt{swish}$ activation function~\citep{swish_act}.
For the initial minimization step in the algorithm,
initial $s_0(t)$ is trained to minimize the analytical relative loss
$$\frac{\int_{\Omega} \| s_0(x)-\nabla \log \rho_0 (x) \|^2 \rho_0(x) dx}{\int_{\Omega} \| \nabla \log \rho_0 (x) \|^2 \rho_0(x) dx}
\approx
\frac{\sum_{i=1}^{N} \| s_0(X_0(x^{(i)}))-\nabla \log \rho_0 (X_0(x^{(i)})) \|^2  }{\sum_{i=1}^{N} \| \nabla \log \rho_0 (X_{0}(x^{(i)})) \|^2 },$$
where $X_0(x^{(i)})=x^{(i)} $ are drawn from the initial Gaussian distribution.
The analytical relative loss is minimized to a tolerance less than $10^{-4}$. 
After the training of $s_0(x)$, we train the neural networks $s_t(x), t>0$ by minimizing the objective function
$ \frac{1}{N} \sum_{i=1}^N \left[\| s_{t}(X_{t}(x^{(i)})) \|^2 + 2 \nabla \cdot s_{t}(X_{t}(x^{(i)}))\right]$.
Here the divergence $\nabla \cdot s_{t}$ is approximated by the denoising score matching loss function \citep{vincent_2011, Boffi}
\begin{equation}\label{eq:denoisingL}
	\begin{aligned}
	    \nabla \cdot s_t(x)
     &=\lim_{\kappa \to 0} \frac{1}{2 \kappa}  \left[\mathbb{E}_{\xi}(s_t(x+ \kappa \xi) \cdot \xi) -\mathbb{E}_{\xi}(s_t(x-\kappa \xi) \cdot \xi) \right]  \\
     &\approx
	\frac{1}{2 \kappa} \left[\mathbb{E}_{\xi}(s_t(x+ \kappa \xi) \cdot \xi) -\mathbb{E}_{\xi}(s_t(x-\kappa \xi) \cdot \xi) \right],
	\end{aligned}
\end{equation}
with $\xi \sim N(0, I), $  and a noise scale $\kappa =1e-4.$
We initialize the parameters $\theta_{t +\Delta t} $ of the neural network $s_{t +\Delta t}$ to those of the trained neural network $s_{t}$.
The time step is $\Delta t=5e-4.$
We train the networks using Adam optimizer with a learning rate of $\eta = 10^{-4},$ 
until the norm of the gradient is below $gtol=0.3$ for the 1-2000 time step, $gtol=0.35$ for the 2000-9000 time step, and $gtol=0.4$ afterwards.

\paragraph{Comparison} 
We quantitatively compare the solution learnt by the MSBTM Algorithm~\ref{alg:msbtm} with the results of numerical SDE integration and the analytical results derived from Gaussian distribution,
using three metrics as in \citep{Boffi}:
the trace of the covariance matrix, relative Fisher divergence
and entropy production rate.
First, we compare the trace of empirical covariance from samples $\left\{X_{t}(x^{(i)}) \right\}_{i=1}^{N}$ obtained by MSBTM algorithm with  $X_{t}(x^{(i)}) \in \mathbb{R}^2$, and those of SDE integration,
the analytical covariance $C_t$ solved from the ODE~\eqref{eq:MC} at different time step $t$.
We also compute the empirical covariance of samples obtained from integration of the dynamics without noise term
\begin{equation}
	\label{eq:noisefree}
	dx^{(i)}_t = (\beta_t - x^{(i)}_t)dt  - 
	\frac{\alpha}{N} \sum_{j=1}^N 
	\left(x^{(i)}_t - x^{(j)}_t \right)dt,\quad i=1, \dots, N.
\end{equation}
We plot the trace along the time in figure~\ref{fig:Qt}.
Second,
we compute the relative Fisher divergence defined by 
\begin{equation}
	\label{eq:discrep}
	\frac{\int_\Omega \|s_t(x) - \nabla\log\rho_t(x) \|^2 \bar{\rho}(x) dx}{\int_\Omega  \| \nabla\log\rho_t(x) \|^2 \bar{\rho}(x)dx},
\end{equation} 
where the target $\nabla \log \rho_t(x) = -C_t^{-1}(x - m_t).$
Here $\bar{\rho}$ can be taken over the particles obtained from the Algorithm~\ref{alg:msbtm} (training data), or the samples calculated by SDE integration (SDE data).
Note that the computation of relative Fisher divergence requires
the prediction of the neural network $s_{t}(\cdot)$ on SDE data, which measures the generalization ability of $s_t$.
We compare the relative Fisher divergence of MSBTM and SDE integration along the time in figure~\ref{fig:Ql}.
Third, entropy production rate $\frac{dE_t}{dt}$ is the time derivative of entropy production 
$E_t = -\int_{\Omega}  \rho_t(x) \log\rho_t(x) dx $.
Numerically it can be computed over the training data:
\begin{equation}
	\frac{d}{dt}E_t = -\int_{\Omega} \nabla \log \rho_t(x)\cdot v_t(x) \rho_t(x)dx 
	\approx -\frac{1}{N} \sum_{i=1}^{N} s_t(X_t(x^{(i)})) \cdot v_t(X_t(x^{(i)}).
\end{equation}
Analytically,
the differential entropy $E_t$ is given by 
\begin{equation}
	\label{eq:St:harm}
	E_t  = \frac{1}{2} \bar d \left(\log\left(2\pi\right) + 1\right) + \frac{1}{2}\log\det C_t,
\end{equation} 
and then entropy production rate is
\begin{equation}
	\frac{d}{dt} E_t = \frac{1}{2}	\frac{d}{dt} \log\det C_t.
\end{equation}
Note that using SDE integration is not able to calculate entropy production rate.
We compare entropy production rate from MSBTM algorithm and theoretical one in figure~\ref{fig:Qe}.
More details are given in \ref{app:qc}.
In addition to the quantitative comparison,	we visualize the trajectory of one of particles along the time in figure~\ref{fig:Trajectory1} for better observation, and the trajectories of 300 particles along the time in figure~\ref{fig:Trajectory}.

\begin{figure}[!hbtp]
	\centering
	\begin{subfigure}[t]{0.49\textwidth}
		\centering
		\includegraphics[width=0.85\textwidth]{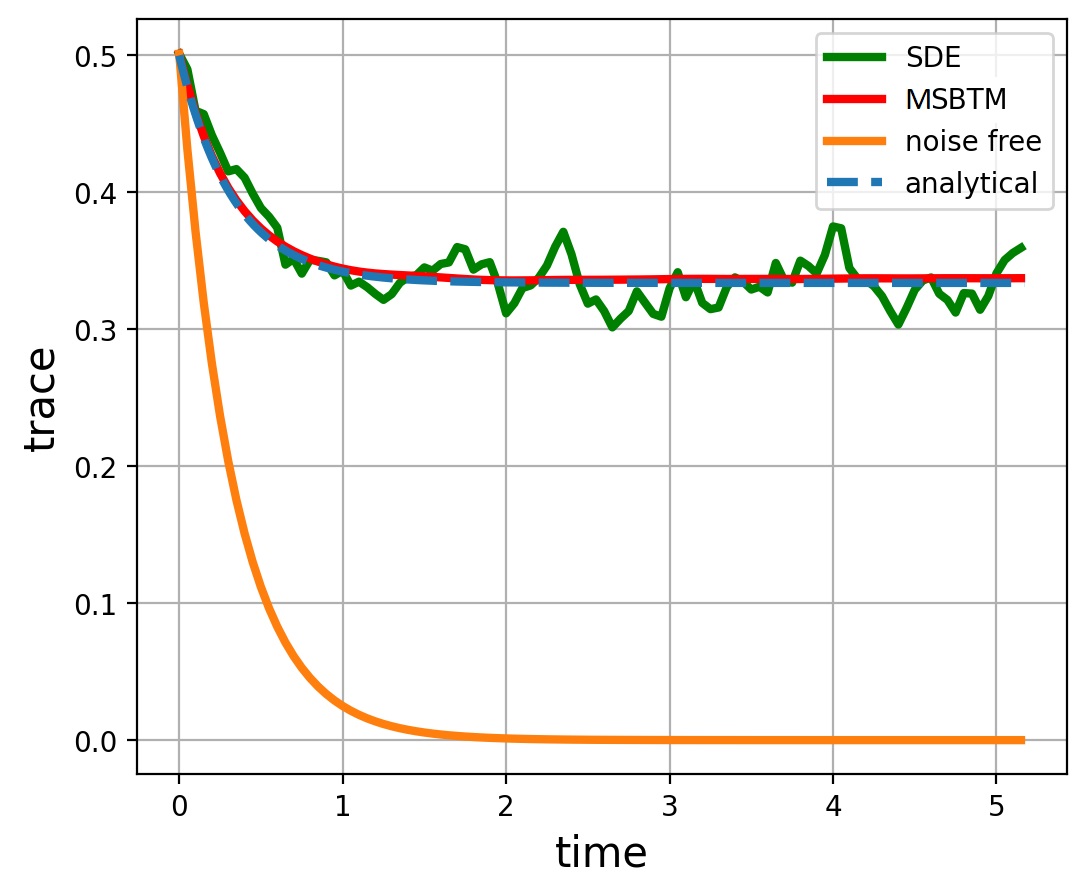}
		\caption{}
		\label{fig:Qt}
	\end{subfigure}
	\begin{subfigure}[t]{0.49\textwidth}
		\centering
		\includegraphics[width=0.85\textwidth]{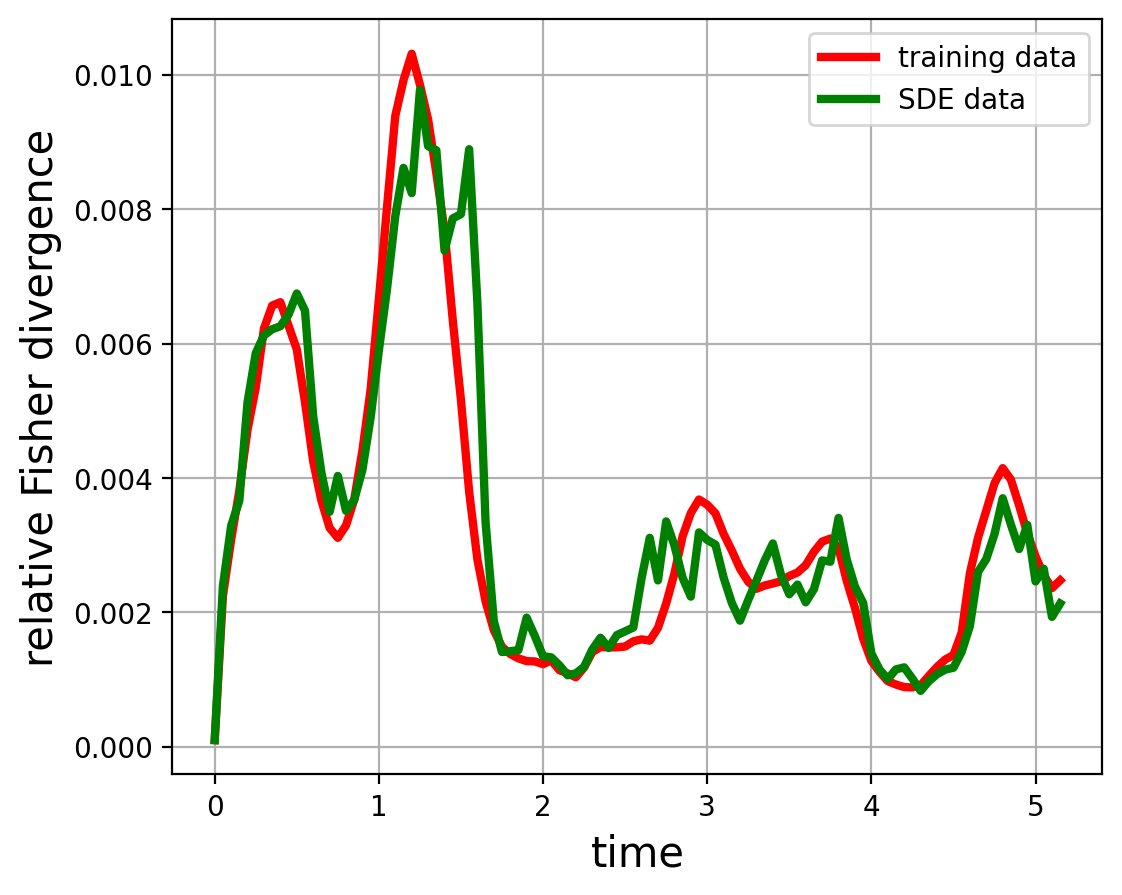}
		\caption{}
		\label{fig:Ql}
	\end{subfigure}
 
	\begin{subfigure}[t]{0.49\textwidth}
		\centering
		\includegraphics[width=0.85\textwidth]{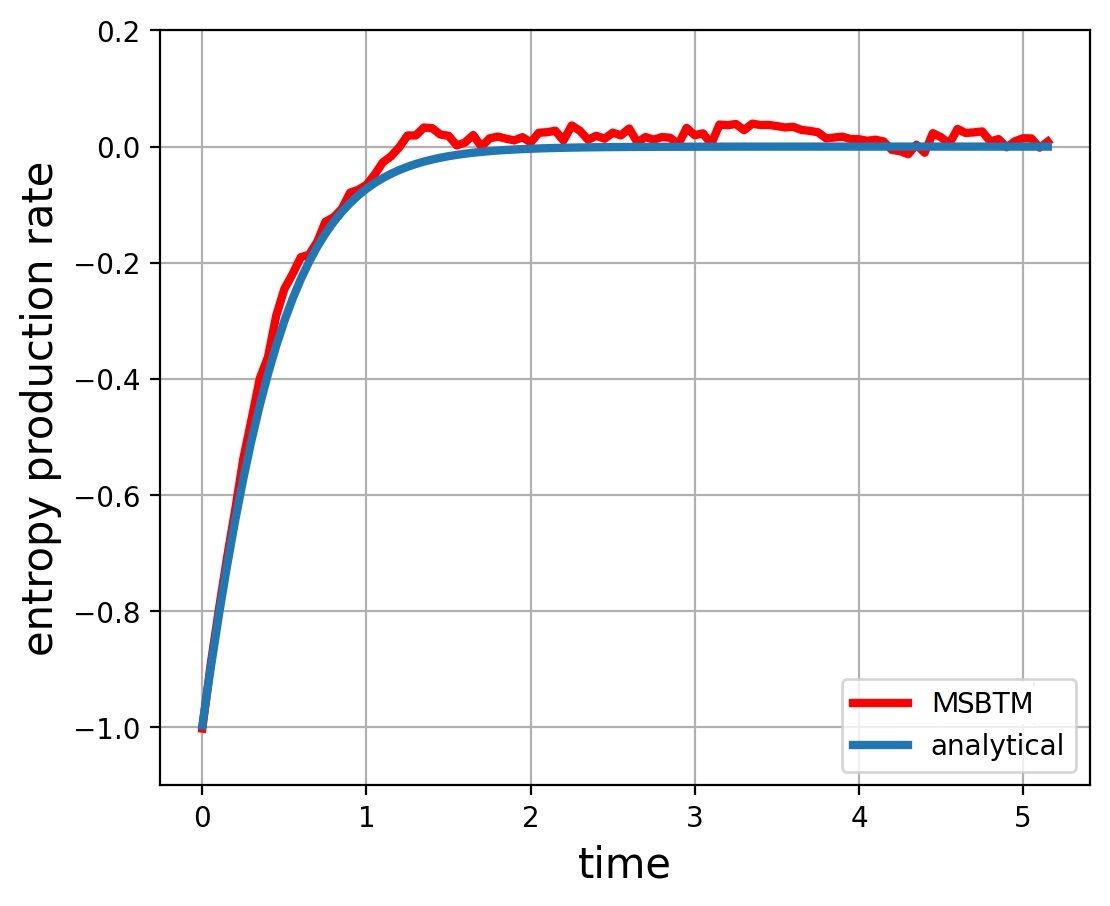}
		\caption{}
		\label{fig:Qe}
	\end{subfigure}	
     \caption{Quantitative comparison  between the numerical solution from MSBTM algorithm, SDE integration and the analytical solution. (a) Trace,  (b) relative Fisher divergence, (c) entropy production rate. }	
     \label{fig:Qthree}
\end{figure}

\begin{figure}[!hbtp]
	\begin{subfigure}[t]{0.32\textwidth}
		\centering
		\includegraphics[width=0.99\textwidth]{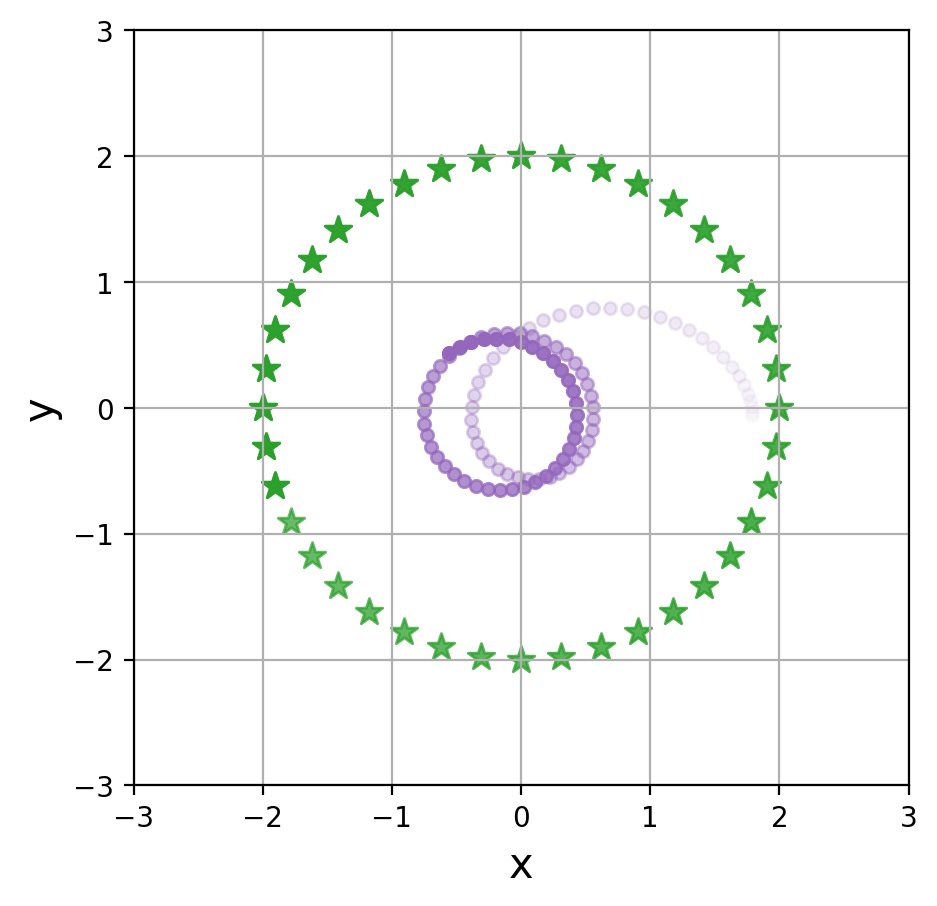}
		\caption{MSBTM}
		\label{fig:SBTM1}
	\end{subfigure}
	\begin{subfigure}[t]{0.32\textwidth}
		\centering
		\includegraphics[width=0.99\textwidth]{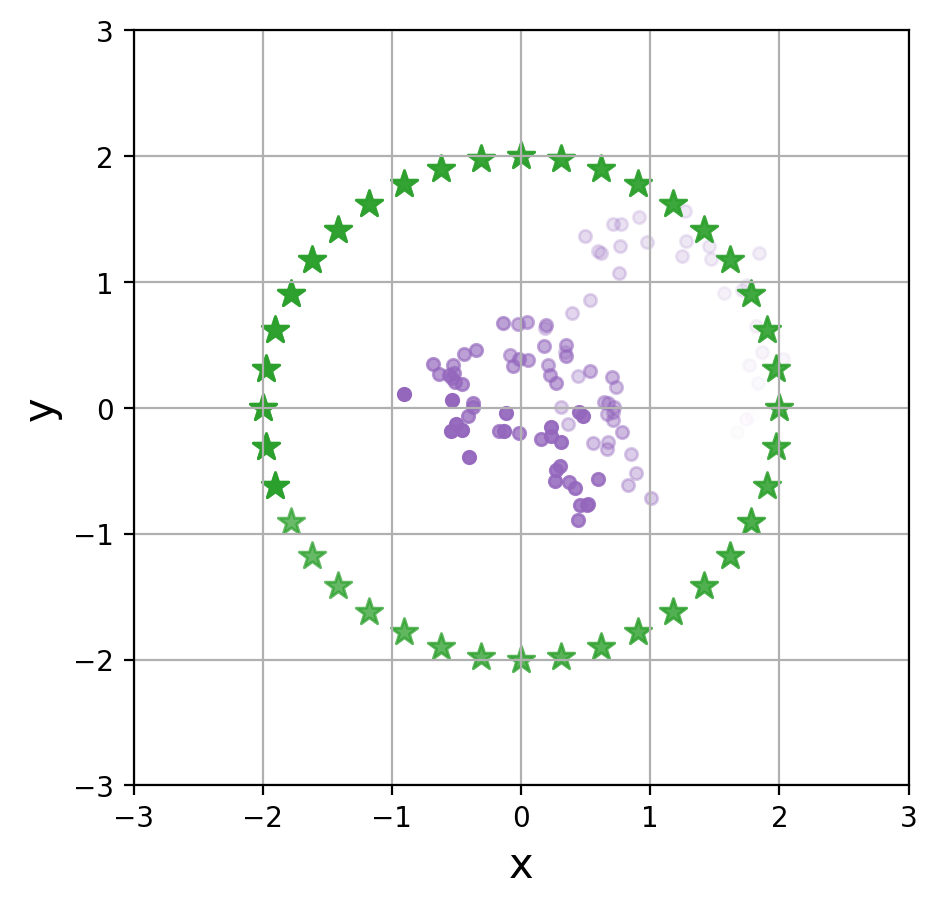}
		\caption{SDE}
		\label{fig:SDE1}
	\end{subfigure}
	\begin{subfigure}[t]{0.32\textwidth}
		\centering
		\includegraphics[width=0.99\textwidth]{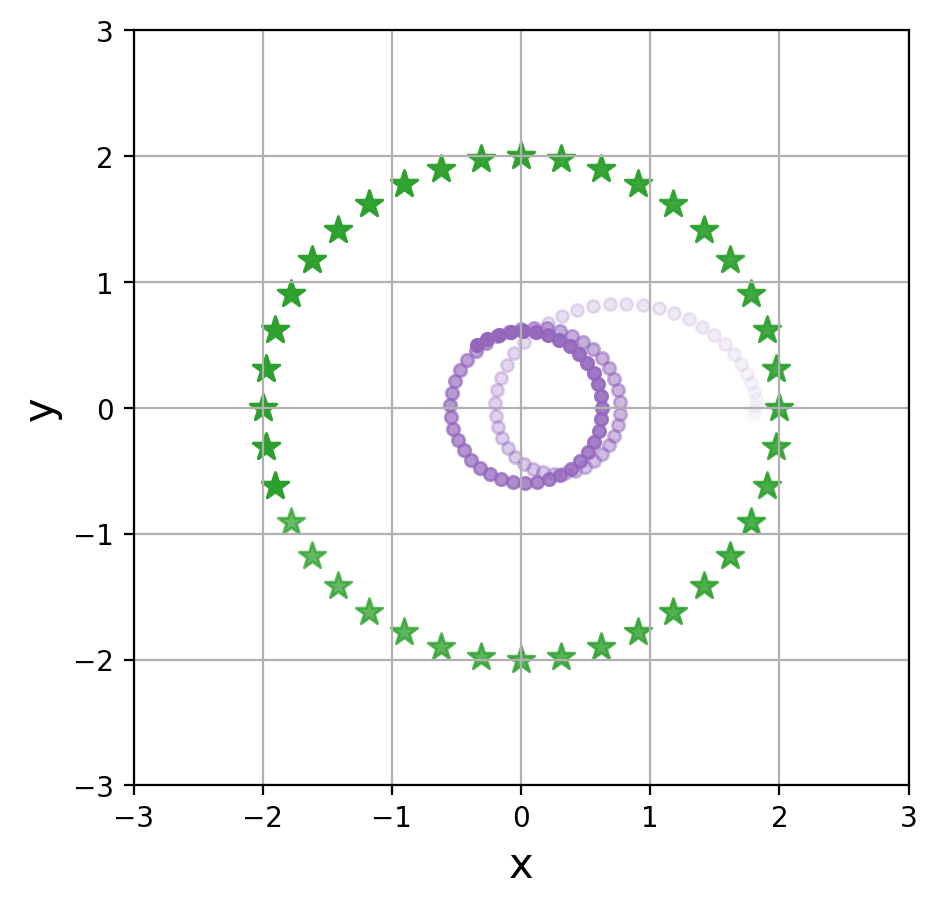}
		\caption{Noise free}
		\label{fig:NO1}
	\end{subfigure}	
	\caption{Trajectory visualization of certain one particle out of $N=300$ particles in a harmonic trap with a harmonic interaction~\eqref{eq:harmonic_SDE}. The mean of the trap $\beta_t$ is shown with a green star and the particles are shown with purple dots. A fading trajectory indicates past positions of the trap mean and the particles. }	
	\label{fig:Trajectory1}
\end{figure}

\paragraph{Results}
Results for quantitative comparison in term of trace, relative Fisher divergence are shown in figure~\ref{fig:Qthree}. 
In figure~\ref{fig:Qt}, the MSBTM algorithm accurately predicts the trace of samples in the system~\eqref{eq:harmonic_SDE1} as the analytical results.
However the SDE predicts the covariance with more fluctuations, and the noise-free system~\eqref{eq:noisefree} incorrectly estimates the covariance as converging to zero.
In figure~\ref{fig:Ql}, the relative Fisher convergence over training data and SDE data match, which demonstrates that the neural networks $s_t(\cdot)$ generalize well on the SDE samples.
In figure~\ref{fig:Qe}, the MSBTM algorithm approximates the theoretical entropy production rate 
while the SDE cannot make a prediction for entropy production rate.
Furthermore, from the trajectories in figure~\ref{fig:Trajectory1} and figure~\ref{fig:Trajectory},
the MSBTM algorithm captures the variance of the stochastic dynamics~\eqref{eq:harmonic_SDE1} as SDE integration, and generates deterministic trajectories similar to that of the noise free system~\eqref{eq:noisefree}. The particles are trapped in the circle determined by $\beta_t, t \geq 0.$

\begin{figure}[!hbtp]
	\begin{subfigure}[t]{0.32\textwidth}
		\centering
		MSBTM
		\includegraphics[width=0.99\textwidth]{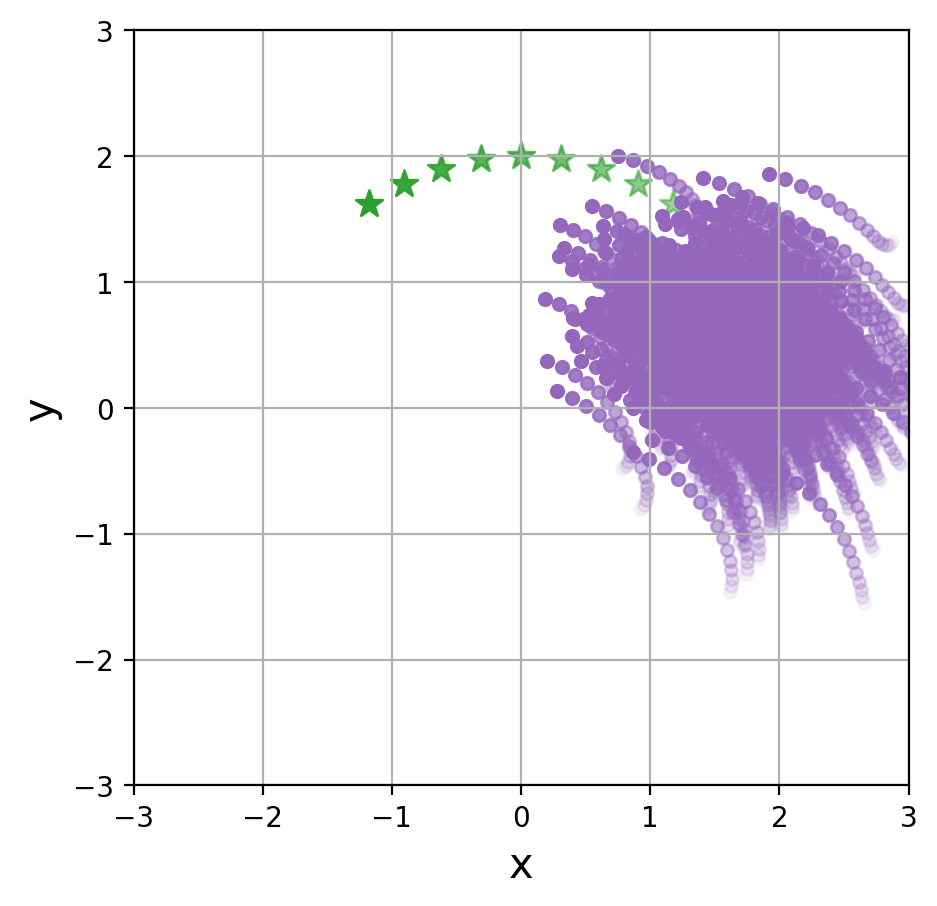}
	\end{subfigure}
	\begin{subfigure}[t]{0.32\textwidth}
		\centering
		SDE
		\includegraphics[width=0.99\textwidth]{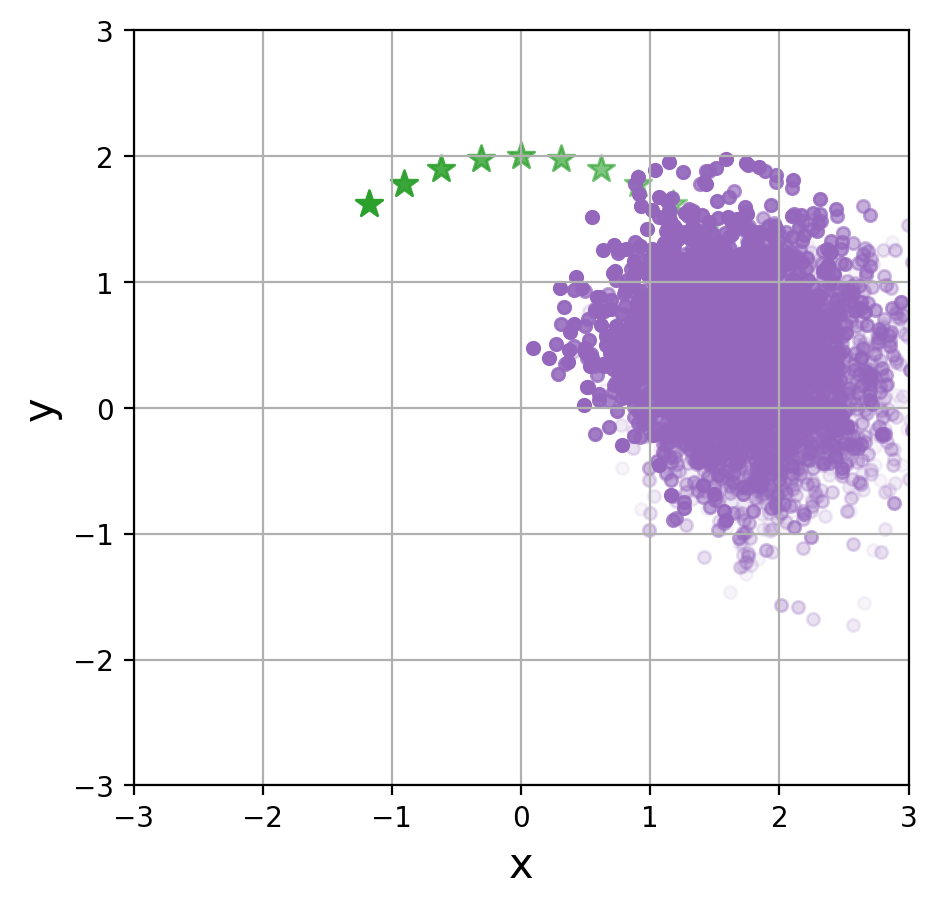}
	\end{subfigure}
	\begin{subfigure}[t]{0.32\textwidth}
		\centering
		Noise free
		\includegraphics[width=0.99\textwidth]{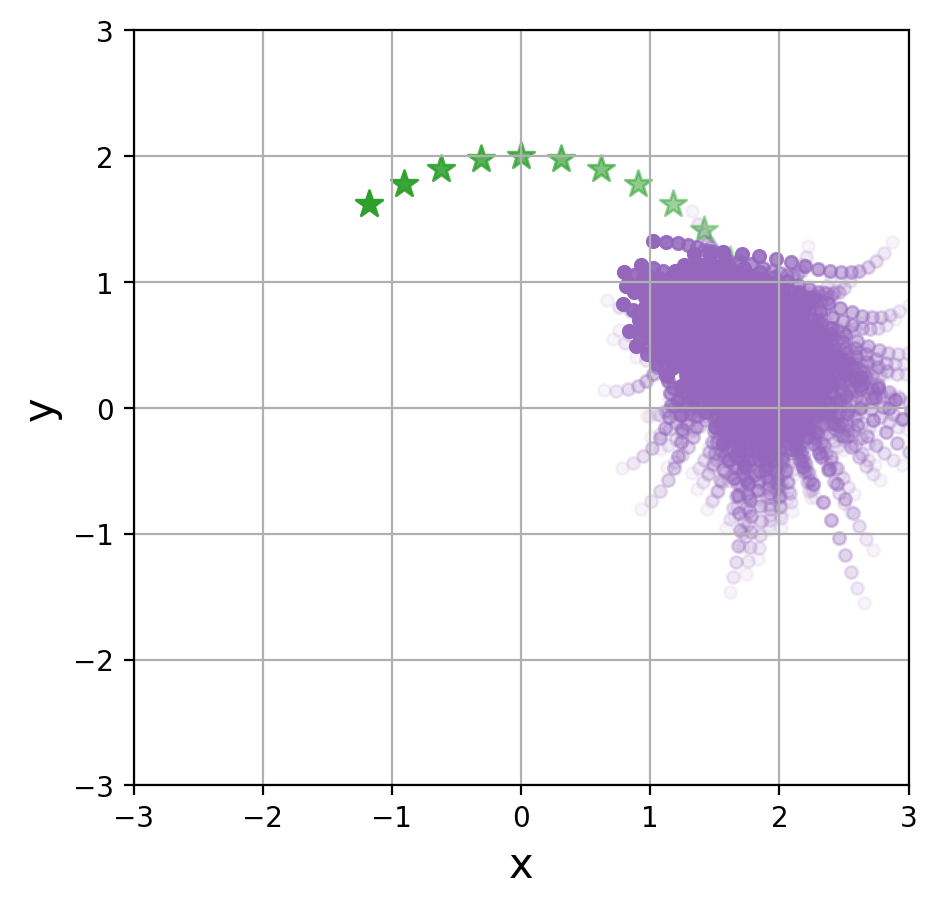}
	\end{subfigure}	

   \begin{subfigure}[t]{0.32\textwidth}
   	\centering
   	\includegraphics[width=0.99\textwidth]{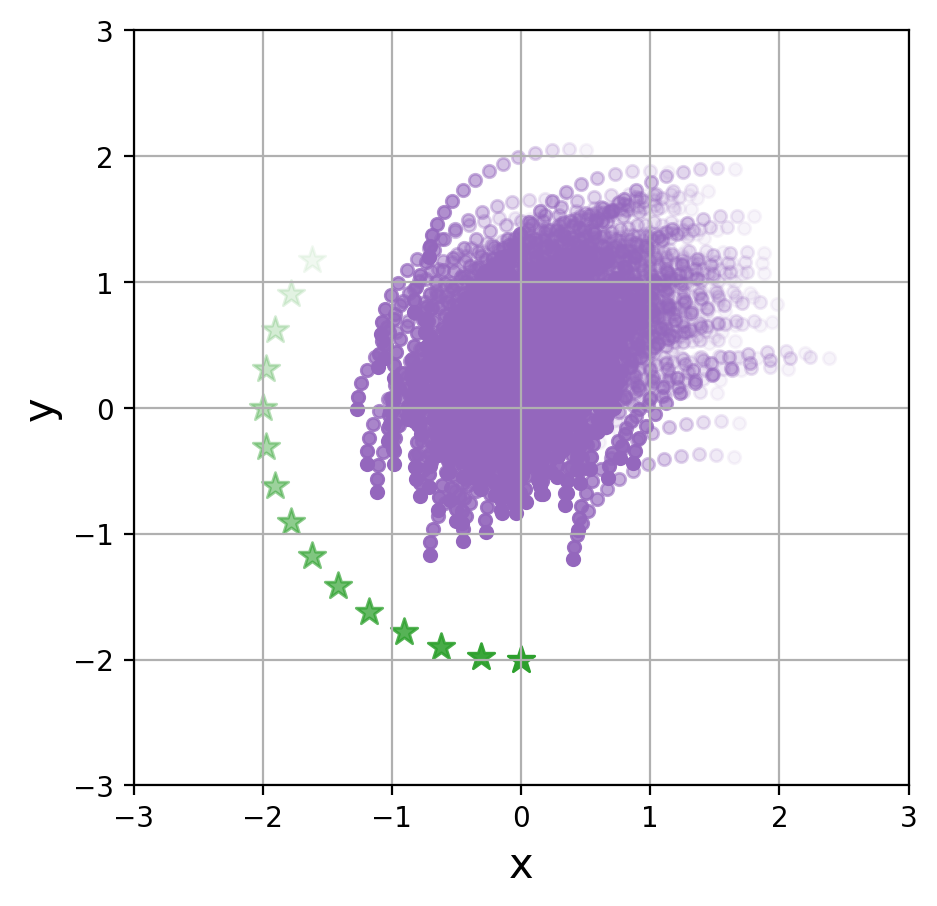}
   \end{subfigure}
   \begin{subfigure}[t]{0.32\textwidth}
   	\centering
   	\includegraphics[width=0.99\textwidth]{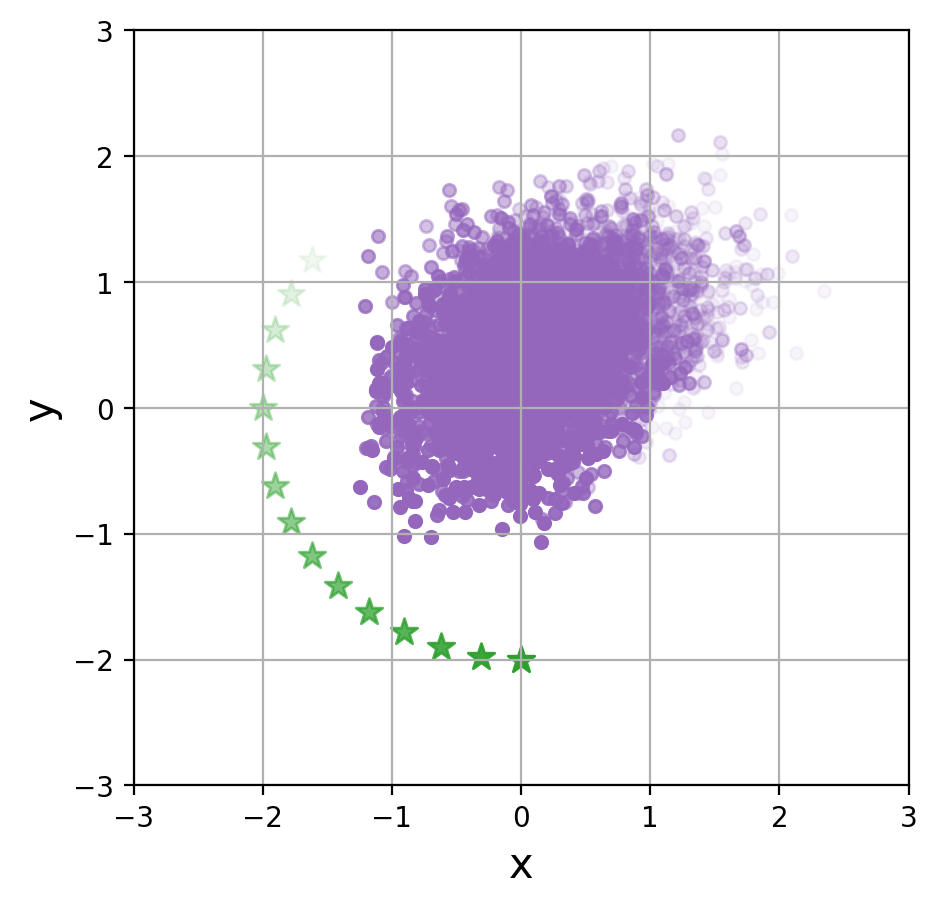}
   \end{subfigure}
   \begin{subfigure}[t]{0.32\textwidth}
   	\centering
   	\includegraphics[width=0.99\textwidth]{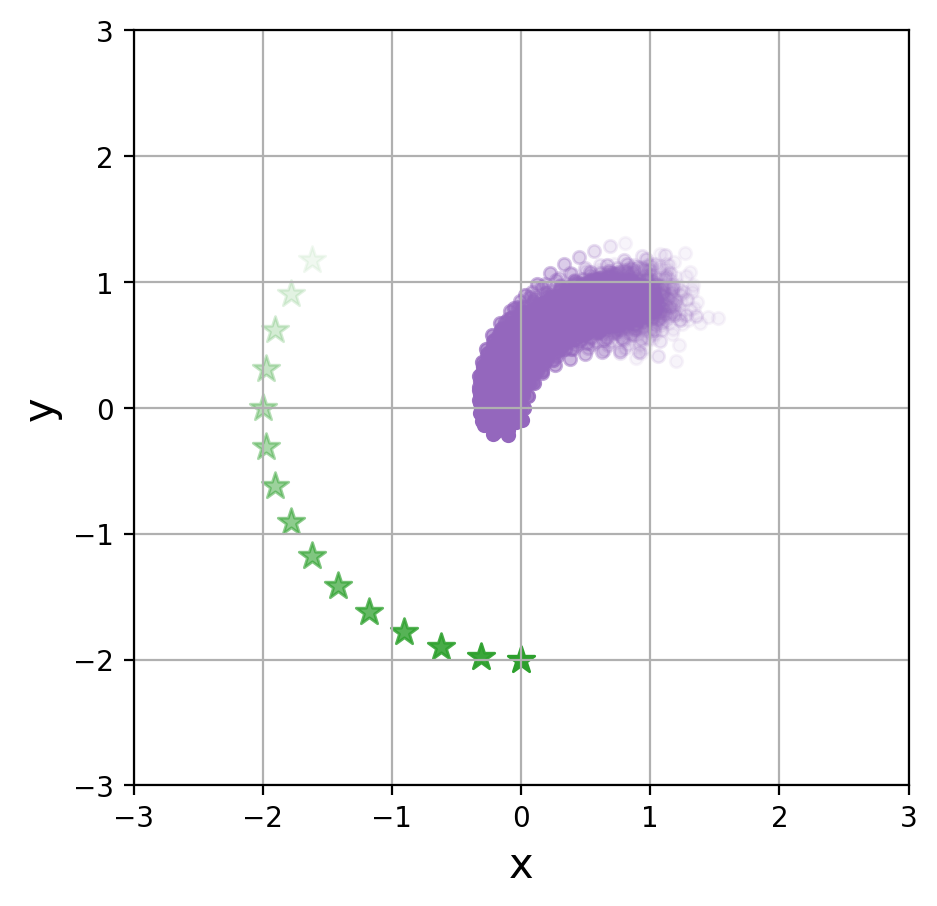}
   \end{subfigure}	

   \begin{subfigure}[t]{0.32\textwidth}
   	\centering
   	\includegraphics[width=0.99\textwidth]{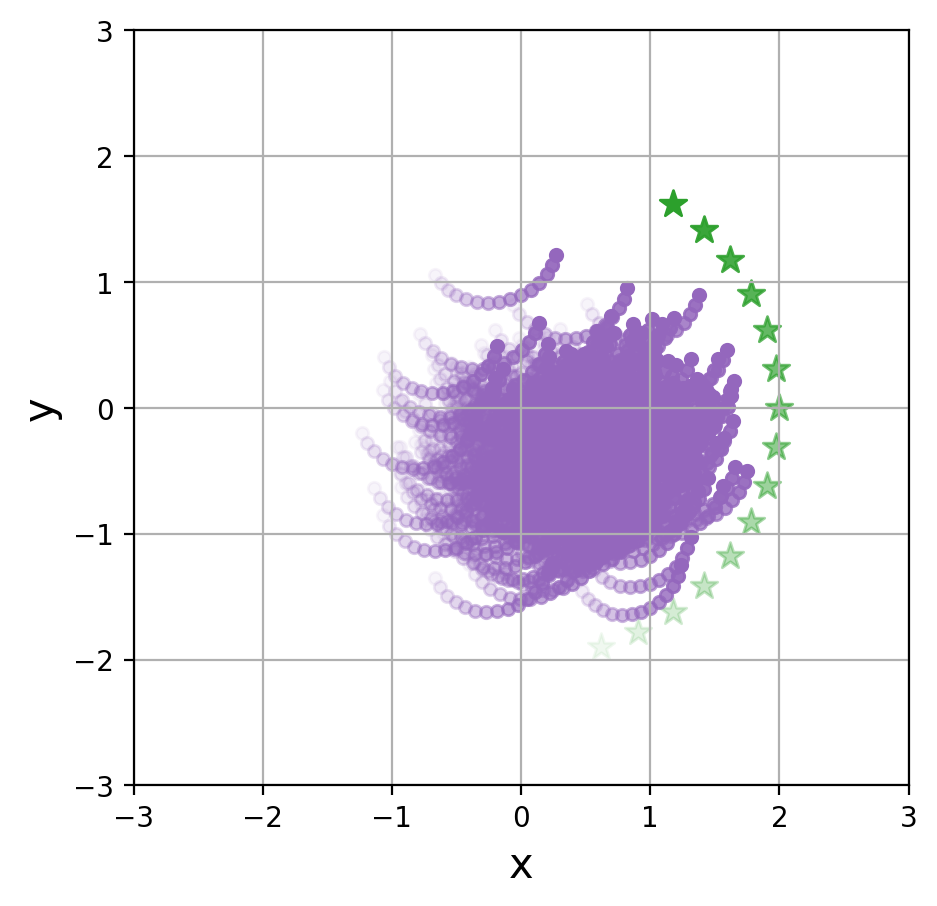}
   \end{subfigure}
   \begin{subfigure}[t]{0.32\textwidth}
   	\centering
   	\includegraphics[width=0.99\textwidth]{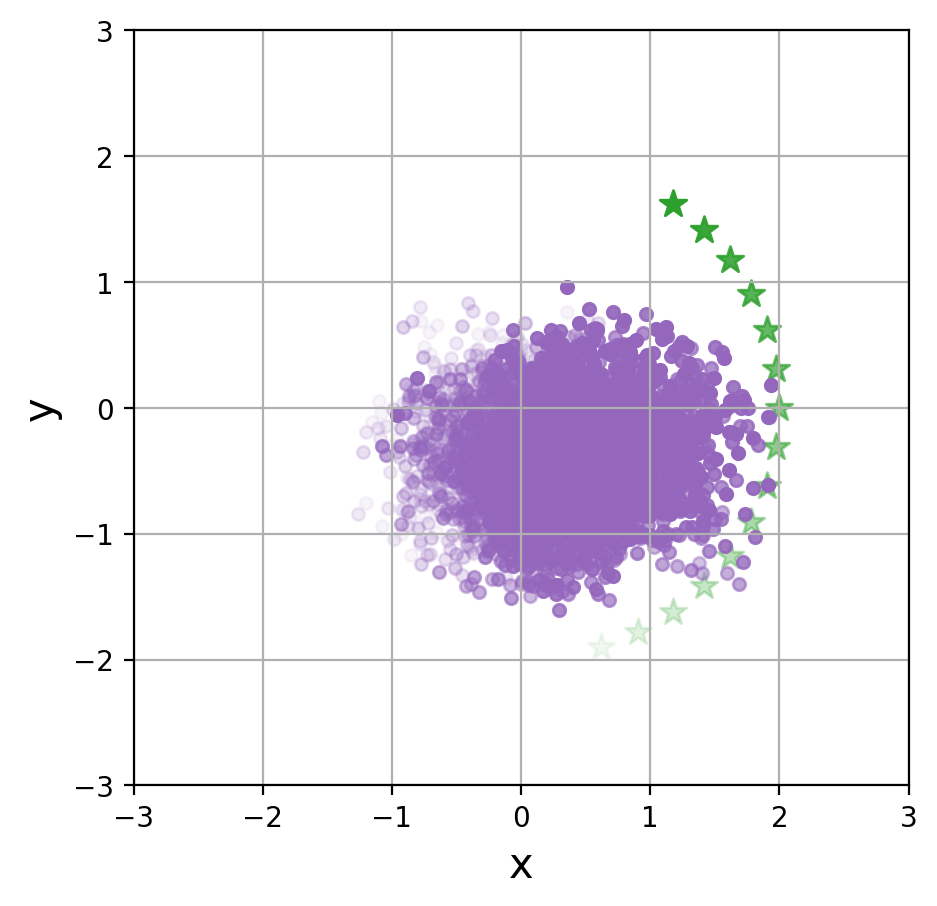}
   \end{subfigure}
   \begin{subfigure}[t]{0.32\textwidth}
   	\centering
   	\includegraphics[width=0.99\textwidth]{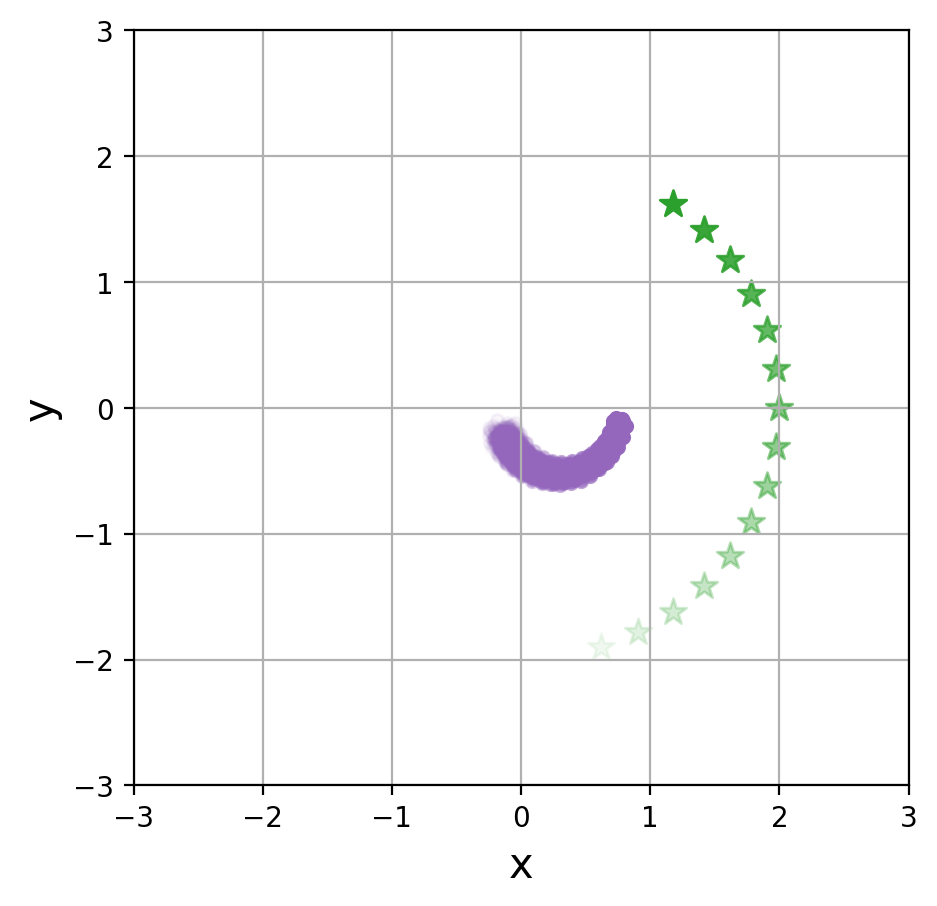}
   \end{subfigure}	

  	\begin{subfigure}[t]{0.32\textwidth}
  	\centering
  	\includegraphics[width=0.99\textwidth]{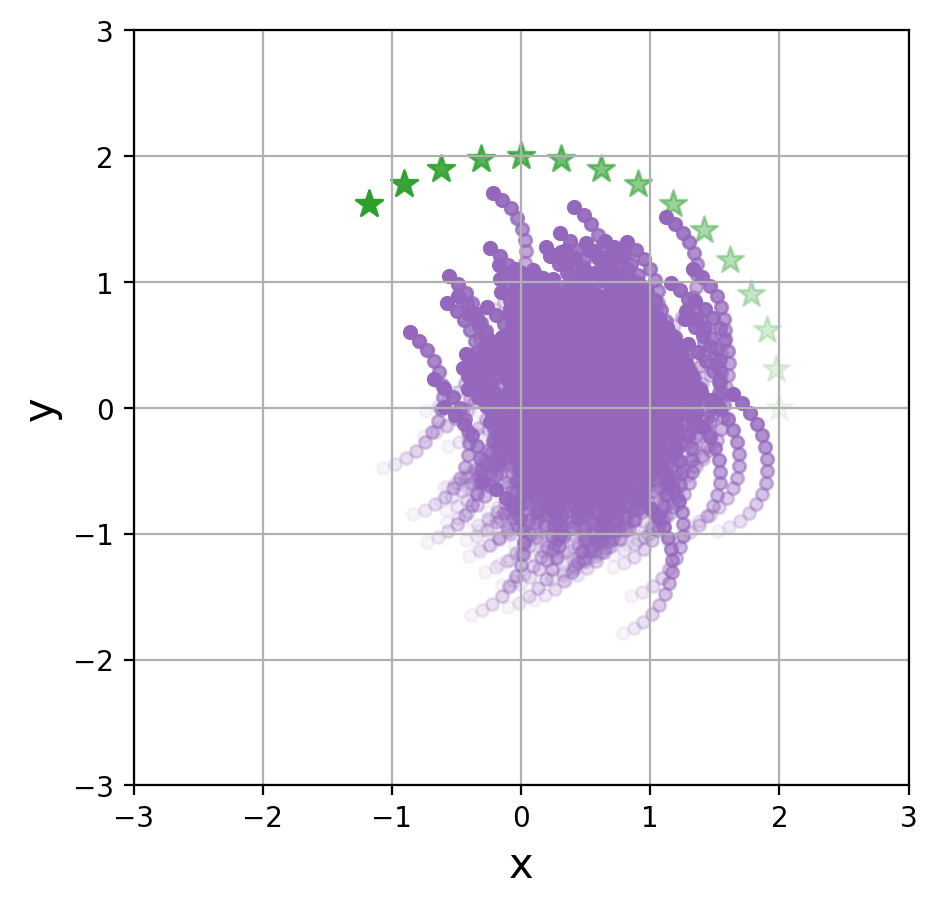}
  \end{subfigure}
  \begin{subfigure}[t]{0.32\textwidth}
  	\centering
  	\includegraphics[width=0.99\textwidth]{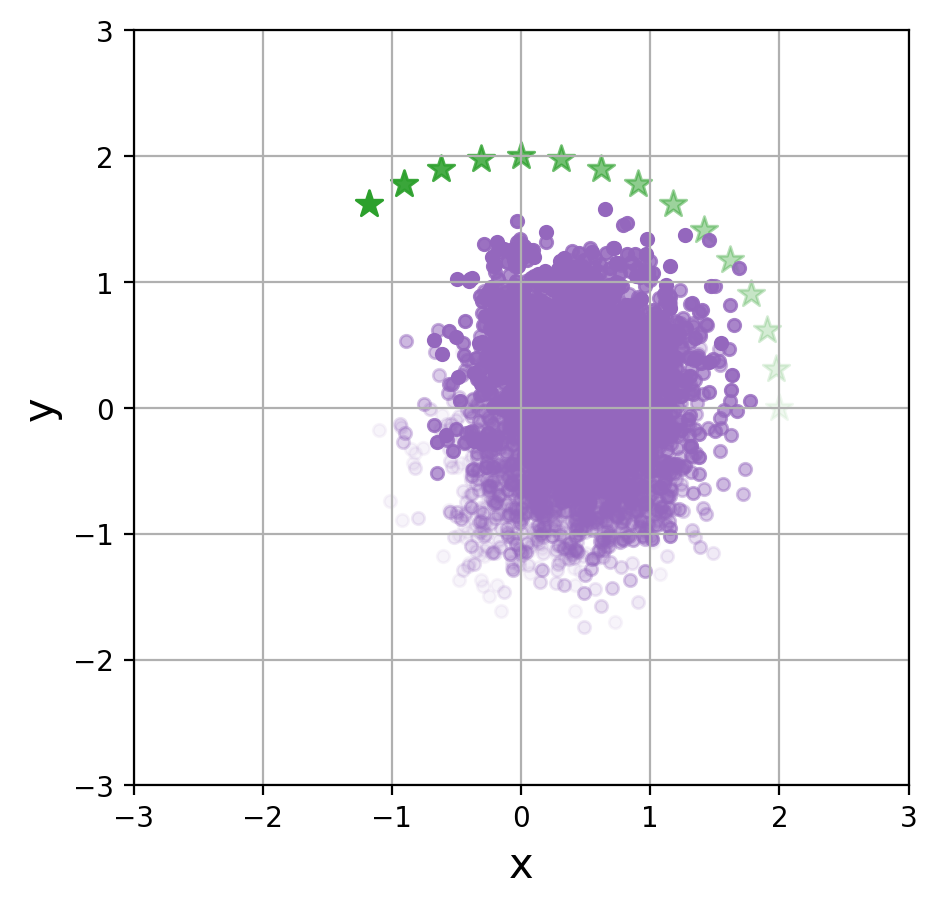}
  \end{subfigure}
  \begin{subfigure}[t]{0.32\textwidth}
  	\centering
  	\includegraphics[width=0.99\textwidth]{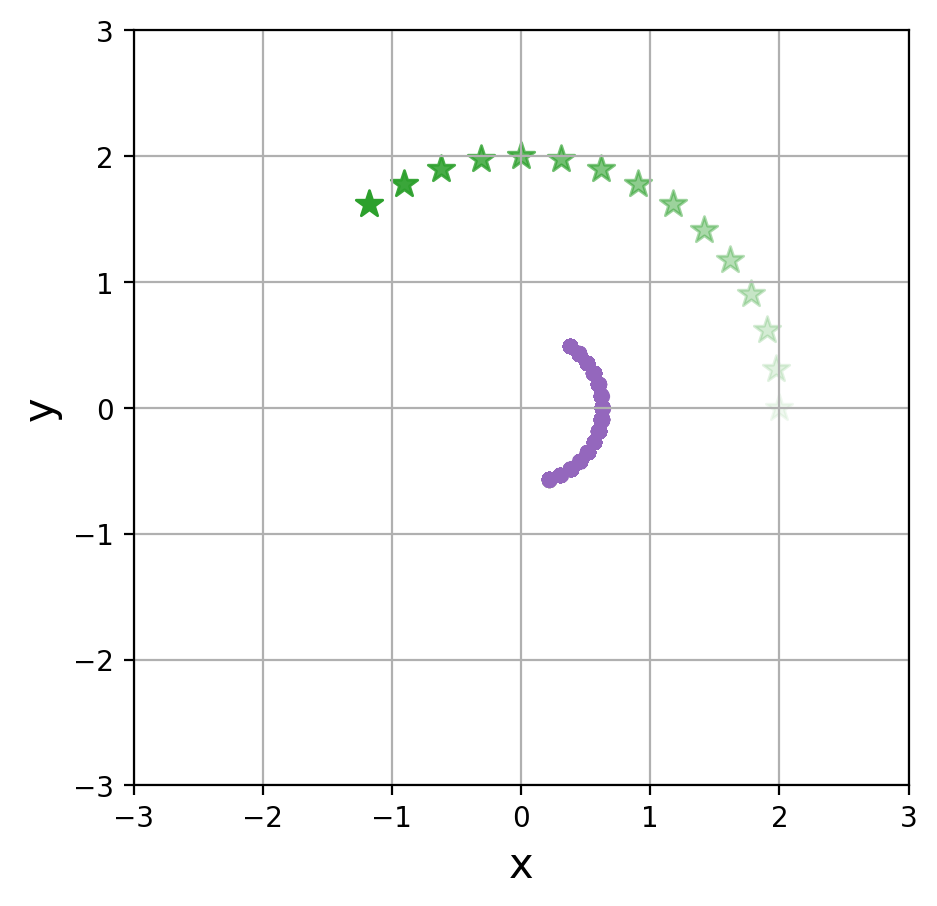}
  \end{subfigure}	
	\caption{Trajectory visualization of $N=300$ particles in a harmonic trap with a harmonic interaction~\eqref{eq:harmonic_SDE}. The mean of the trap $\beta_t$ is shown with a green star and the particles are shown with purple dots. A fading trajectory indicates past positions of the trap mean and the particles. From top to bottom, we present the trajectory along the time.}	
	\label{fig:Trajectory}
\end{figure}

\FloatBarrier
\subsection{Active matter with mean field interaction}
\label{sec:ex2}

In this subsection, we study the mean-field Fokker-Planck equation defined by 
\begin{equation}
\label{eq:swimmerFPE}
  \begin{aligned}
     & \frac{\partial}{\partial t} \rho_t(x, v) = -\frac{\partial}{\partial x} \cdot 
	\left[ 	\left( -x^3 + v \right) \rho_t(x,v) 
	\right]   \\
     & -\frac{\partial}{\partial v} \cdot 
	\left[ 	\left( -\gamma v -  \alpha  \iint_{\Omega} (x -y) \rho_t(y,v) dy dv \right) \rho_t(x, v) 	\right] 
   + \gamma D \frac{\partial^2}{\partial v \partial v} \rho_t(x,v), 
  \end{aligned}
\end{equation}
with $(x,v) \in \Omega \subset \mathbb{R}^{2}$ and parameters $\gamma, \alpha \in \mathbb{R}$.
Correspondingly, there is a two-dimensional system of $N$ particles
and for each particle, the position $x^{(i)} $ and velocity $v^{(i)}$ are described by the SDE
\begin{equation}
	\label{eq:swimmer}
	\begin{aligned}
		dx^{(i)} &= \left(-( x^{(i)})^3 + v_i \right)dt,\\
		dv^{(i)} &= -\gamma v^{(i)} dt -\frac{\alpha}{N} \sum_{j=1}^N (x^{(i)} - x^{(j)}) dt + \sqrt{2\gamma D}dW_t.
	\end{aligned}
\end{equation}
As the number of particles $N$ goes to infinity, their position-velocity pair admits the probability density function $\rho_t(x,v)$ that solves the equation~\eqref{eq:swimmerFPE}.
The system~\eqref{eq:swimmer} generalizes the ``active swimmer'' model in \citep{Boffi} that describes the motion of a particle with a preference to travel in a noisy direction, adding a mean field interaction term with a coefficient $\alpha$.

\paragraph{Setting}
We set $\alpha = 0.5, \gamma = 0.1, D = 1.0, N=5000.$
In the system, the stochastic noise occurs only on the velocity variable $v $ in~\eqref{eq:swimmer}, hence the score $\frac{\partial}{\partial v} \rho_t(x,v)$ is a scalar function.
We parameterize the score directly $s_t:  \mathbb{R}^{2}\rightarrow \mathbb{R}$ using a three-hidden-layer neural network with $\mathtt{n\_hidden} = 32$ neurons per hidden layer, and  $\texttt{swish}$ activation function.
For the minimization in the algorithm,
initial $s_0(x,v)$ is trained to minimize the analytical relative loss as in section~\ref{sec:ex1}
with an initial Gaussian distribution $\rho_0 (x) = N(0, \sigma_0^2 I), \sigma_0=1.$
The analytical relative loss is minimized to a tolerance less than $10^{-4}$. 
After that,	we train the neural network $s_t(x, v), t>0$ by minimizing 
$ \frac{1}{N} \sum_{i=1}^N \left[\vert s_t(X_{t}(x^{(i)}, v^{(i)} )) \vert^2 + 2 \nabla \cdot s_t(X_{t}(x^{(i)}, v^{(i)}))\right]$  
until the norm of the gradient is less than 1 or the number of gradient descents is greater than 3.
Similarly we use denoising loss~\eqref{eq:denoisingL} to approximate the divergence $\frac{\partial}{\partial v}  s_t(x,v).$ 
The time step is $\Delta t=5e-4.$
We initialize the parameters $\theta_{t+\Delta t} $ of the neural network $s_{t+\Delta t}$ to the parameters of $s_t$.
In the training, the Adam optimizer is used with a learning rate of $\eta = 10^{-4}.$

\paragraph{Comparison}
To verify the numerical solutions from the MSBTM algorithm,
we calculate the total variation between the empirical distributions of 5000 samples obtained from the MSBTM algorithm and SDE integration,
since there is no analytical solution to the mean-field Fokker-Planck equation
~\eqref{eq:swimmerFPE}.
Specifically, we discretize the domain $\Omega \subseteq \left[-3,3\right] \times \left[-3,3\right]$ into grids $\left\{\Delta_{i,j}\right\}_{i,j=1}^{12}$ with scale $0.5 \times 0.5$, compute the discrete distribution $P: \left\{\Delta_{i,j}\right\}\to \mathbb{R} $ defined by
\begin{equation}
    P(\Delta_{i,j})= \frac{\# (samples \in \Delta_{i,j})}{\# samples},
\end{equation}
then we calculate the total variation between two discrete probability distributions from samples of the MSBTM algorithm and the SDE integration using
\begin{equation}
    TV( P^{MSBTM}, P^{SDE} )= \sum_{i,j} |P^{MSBTM} (\Delta_{i,j})-P^{SDE} (\Delta_{i,j})|.
\end{equation}
Results for this quantitative comparison are shown in figure~\ref{fig:TV}.
Furthermore, the sample distribution at different time are presented in figure~\ref{fig:Samples}.
We present the results of the MSBTM algorithm, SDE integration and the noise free version with $D_t(x)=0$ in equation~\eqref{eq:swimmer} for comparison.

\begin{figure}[!hbtp]
	\centering
	\includegraphics[width=0.9\textwidth]{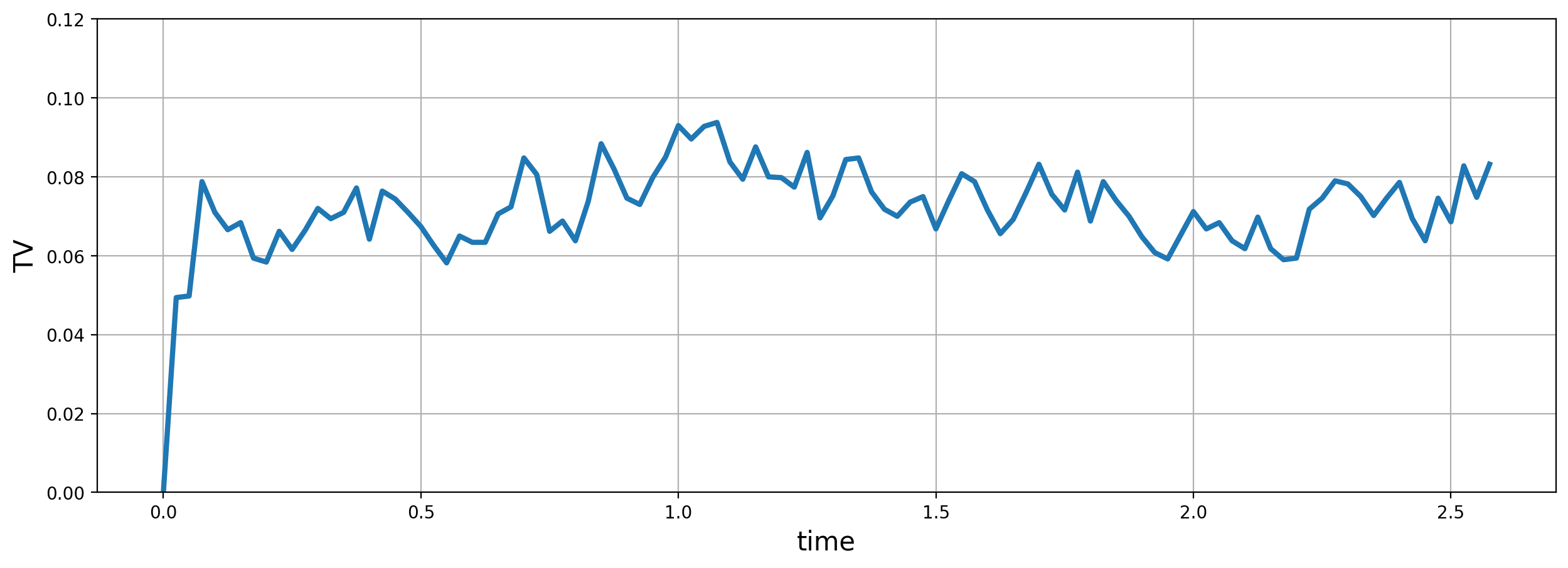}
	\caption{Total variation between 5000 samples of the MSBTM algorithm and those from SDE integration along the time with time step $\Delta t=5e-4$.}
	\label{fig:TV}
\end{figure}

\paragraph{Results} 
Both total variation and samples distribution comparisons demonstrate that
the MSBTM algorithm learns the samples with distribution consistent with that of SDE integration, while the noise-free system grows increasingly and overly compressed with time. 

\begin{figure}[!hbtp]
	\begin{subfigure}[t]{0.32\textwidth}
		\centering
		MSBTM
		\includegraphics[width=0.99\textwidth]{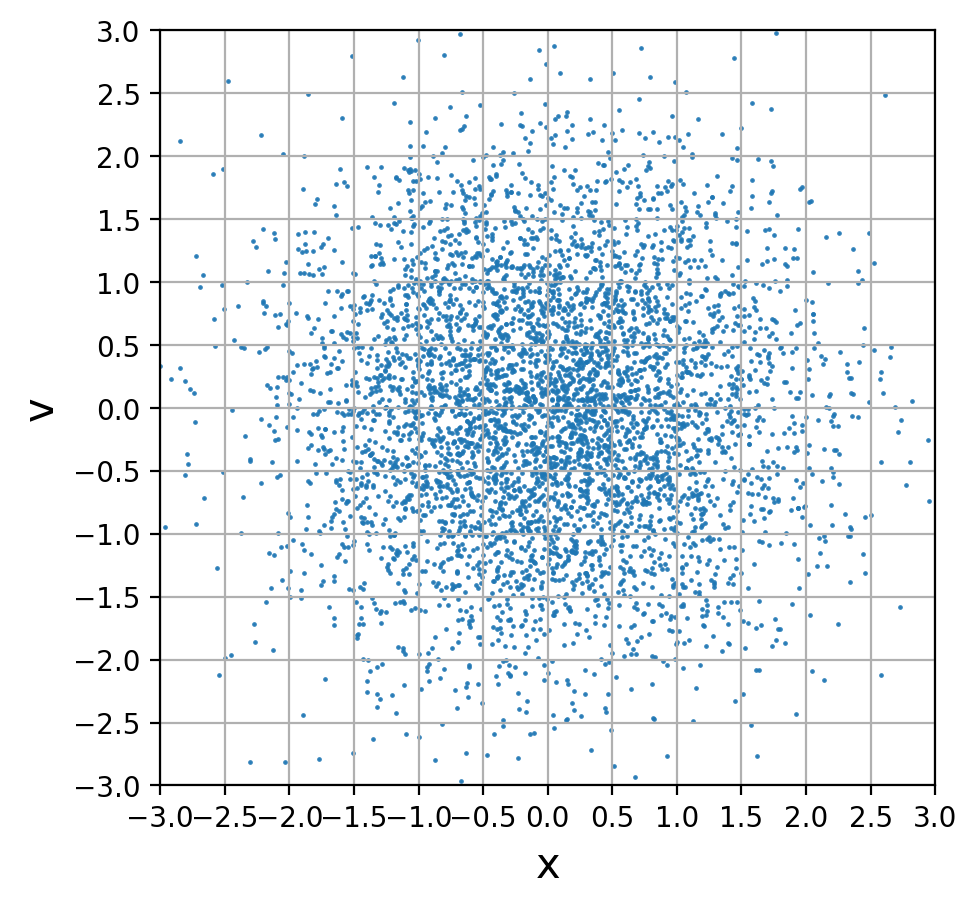}
	\end{subfigure}
	\begin{subfigure}[t]{0.32\textwidth}
		\centering
		SDE
		\includegraphics[width=0.99\textwidth]{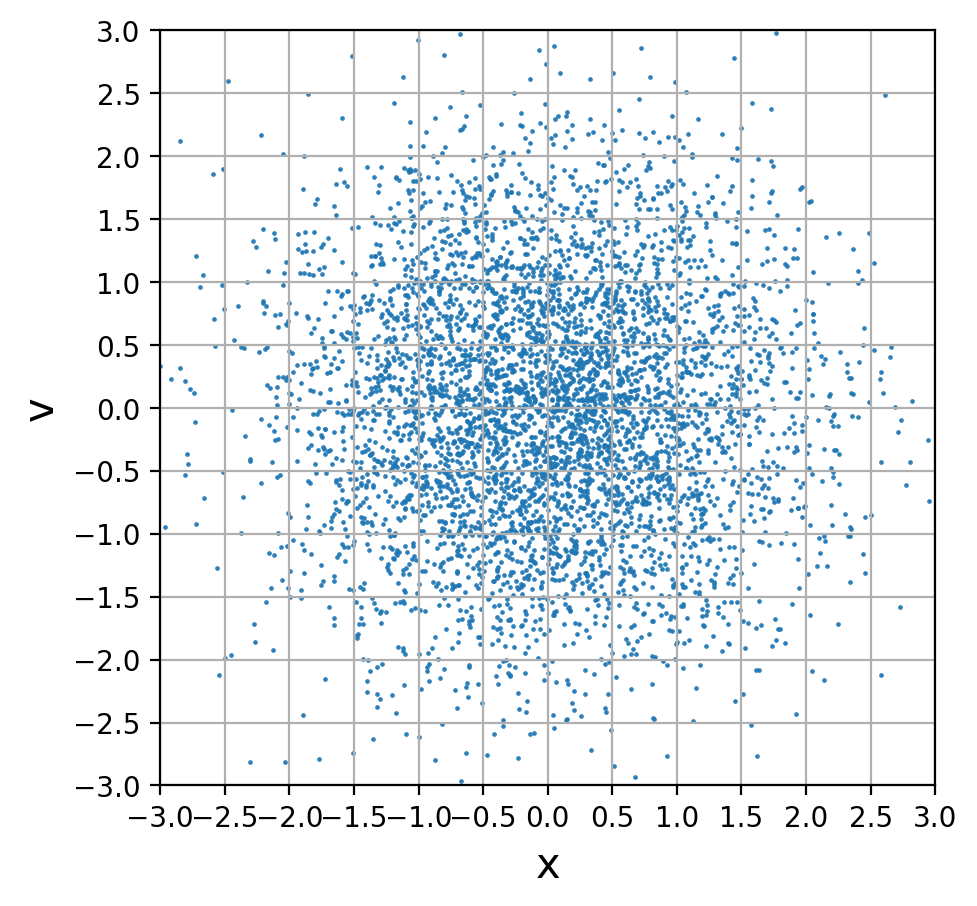}
	\end{subfigure}
	\begin{subfigure}[t]{0.32\textwidth}
		\centering
		Noise free
		\includegraphics[width=0.99\textwidth]{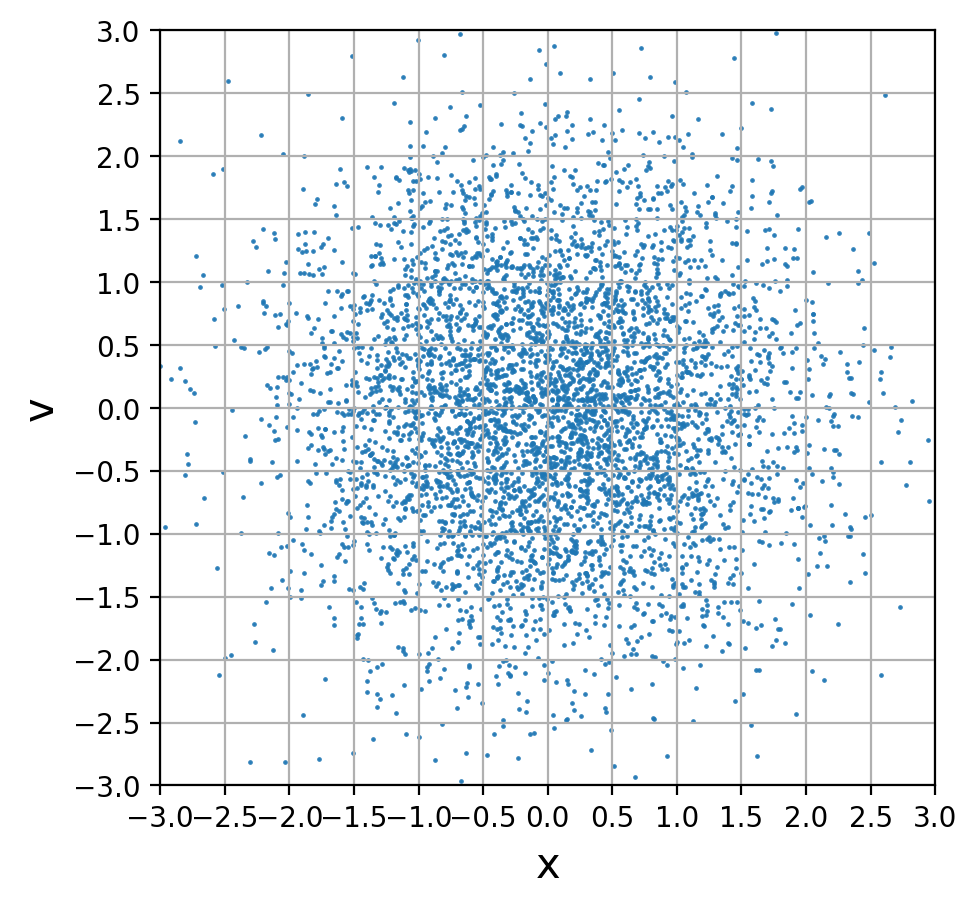}
	\end{subfigure}	
	
	\begin{subfigure}[t]{0.32\textwidth}
		\centering
		\includegraphics[width=0.99\textwidth]{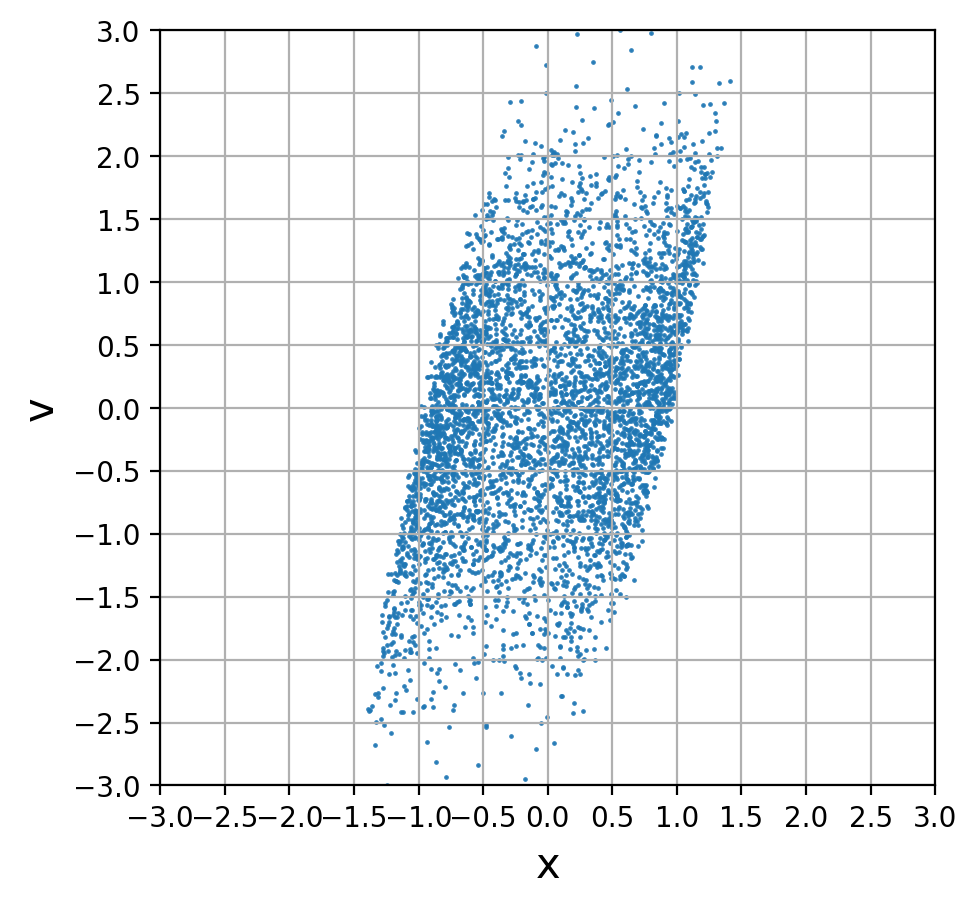}
	\end{subfigure}
	\begin{subfigure}[t]{0.32\textwidth}
		\centering
		\includegraphics[width=0.99\textwidth]{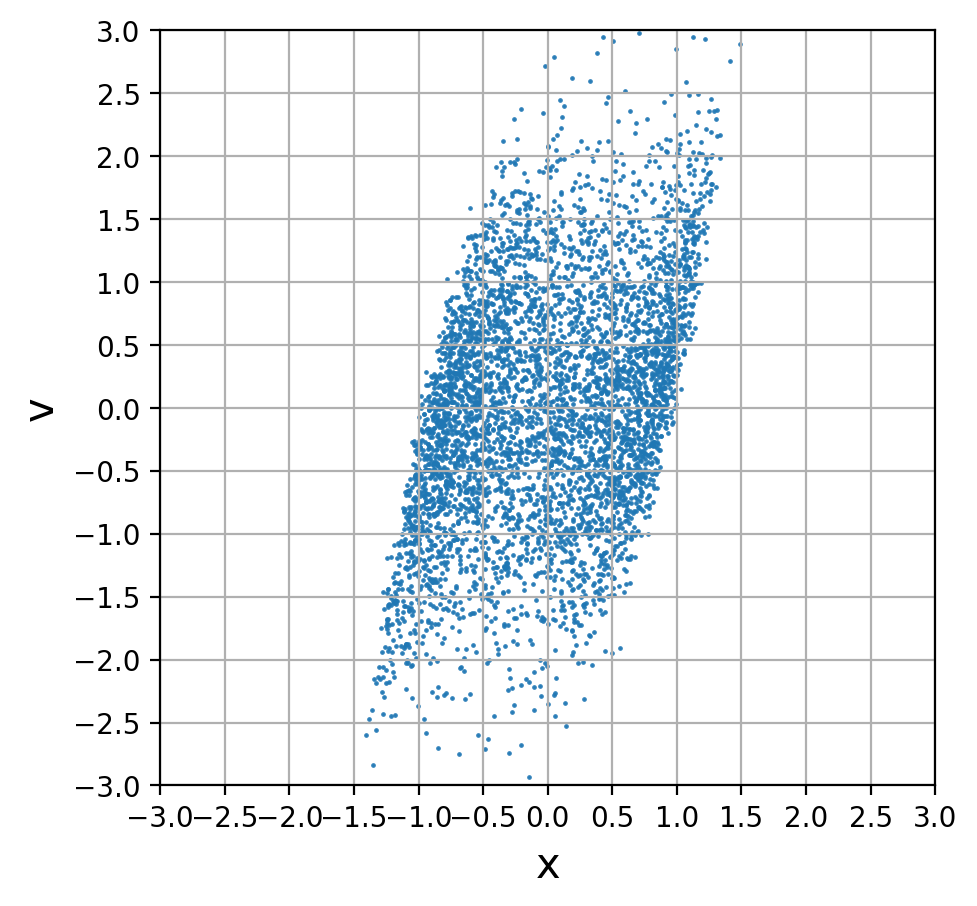}
	\end{subfigure}
	\begin{subfigure}[t]{0.32\textwidth}
		\centering
		\includegraphics[width=0.99\textwidth]{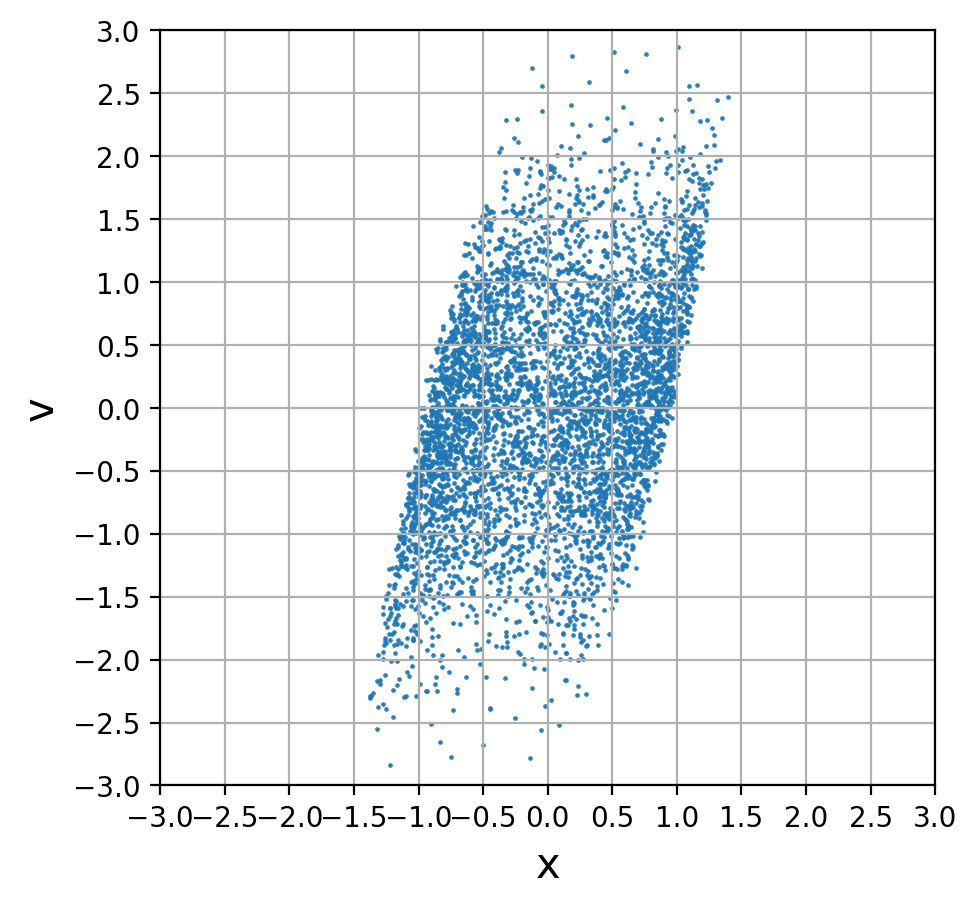}
	\end{subfigure}	
	
	\begin{subfigure}[t]{0.32\textwidth}
		\centering
		\includegraphics[width=0.99\textwidth]{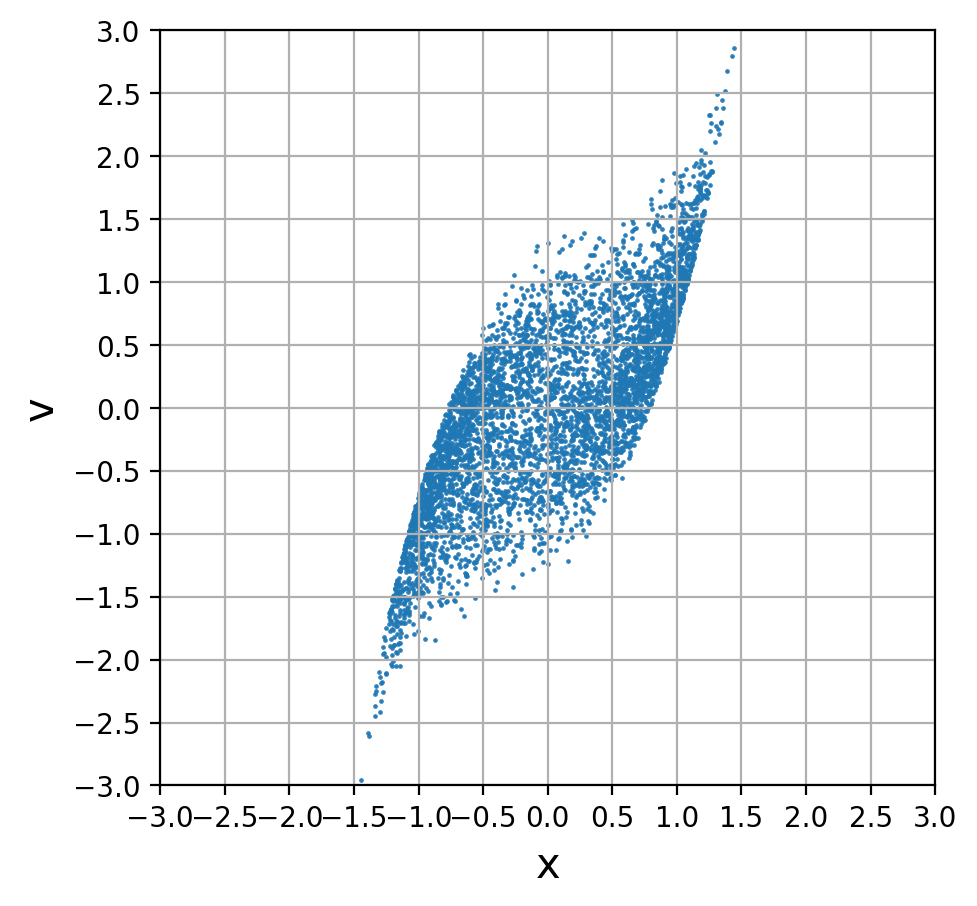}
	\end{subfigure}
	\begin{subfigure}[t]{0.32\textwidth}
		\centering
		\includegraphics[width=0.99\textwidth]{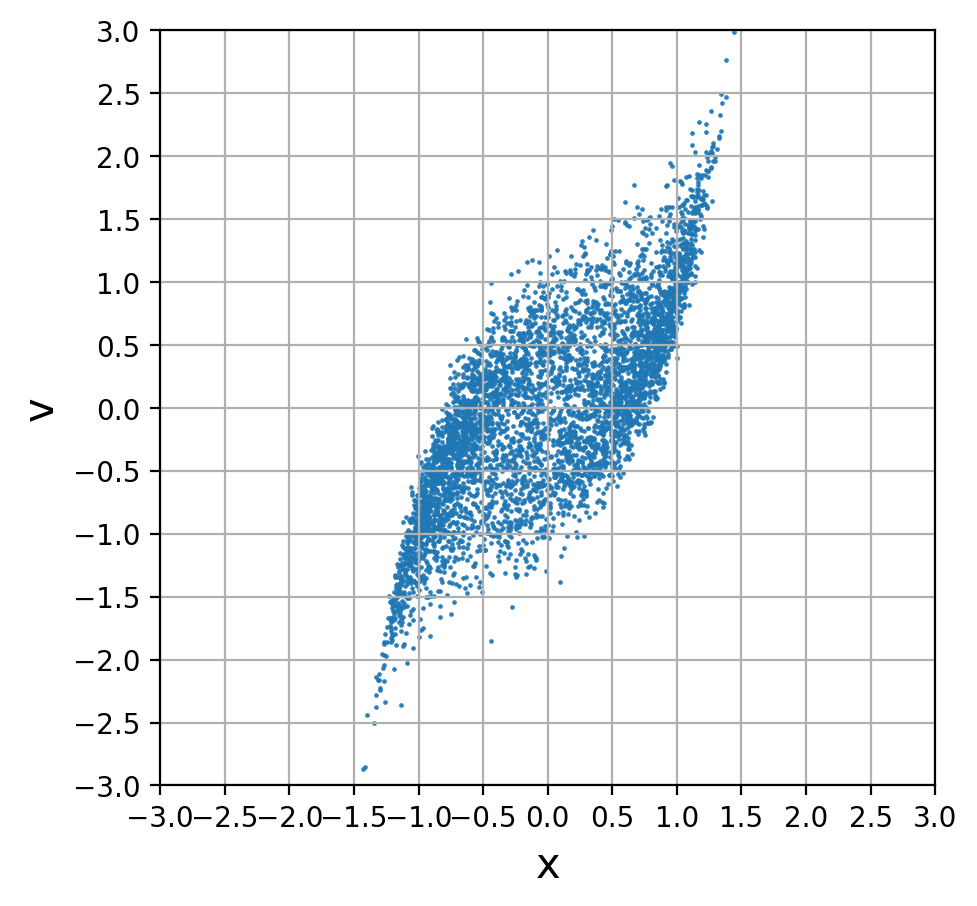}
	\end{subfigure}
	\begin{subfigure}[t]{0.32\textwidth}
		\centering
		\includegraphics[width=0.99\textwidth]{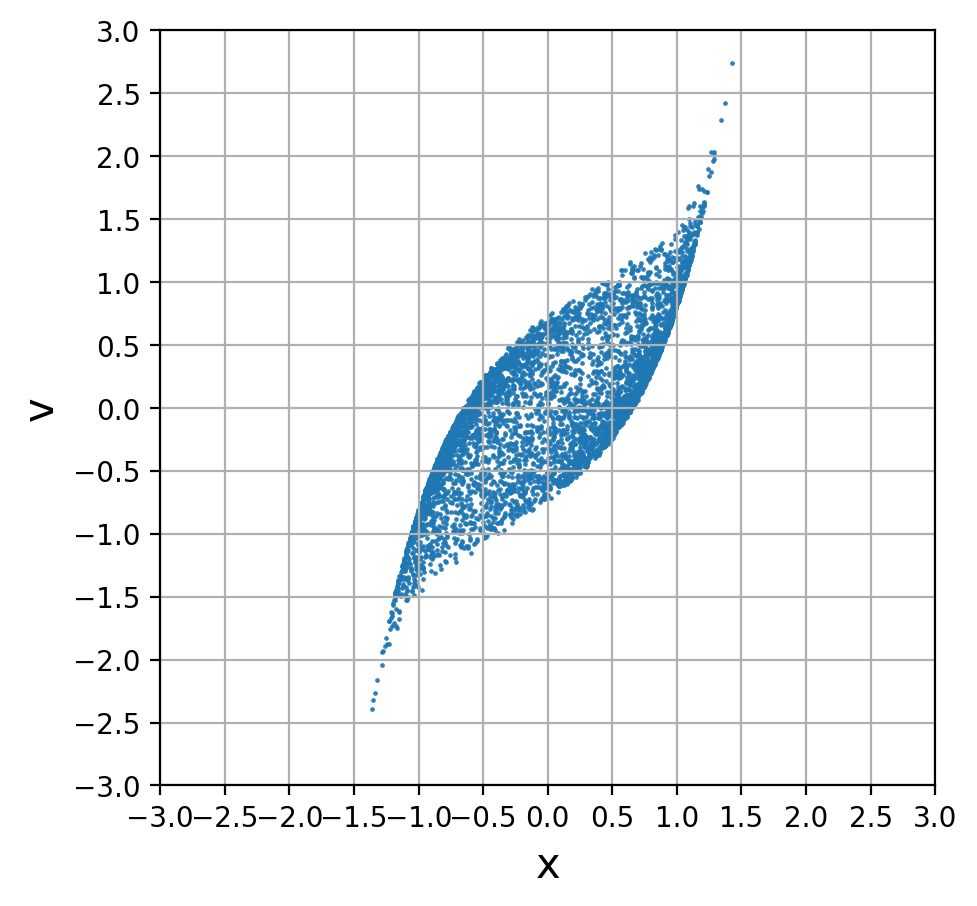}
	\end{subfigure}	
	
	\begin{subfigure}[t]{0.32\textwidth}
		\centering
		\includegraphics[width=0.99\textwidth]{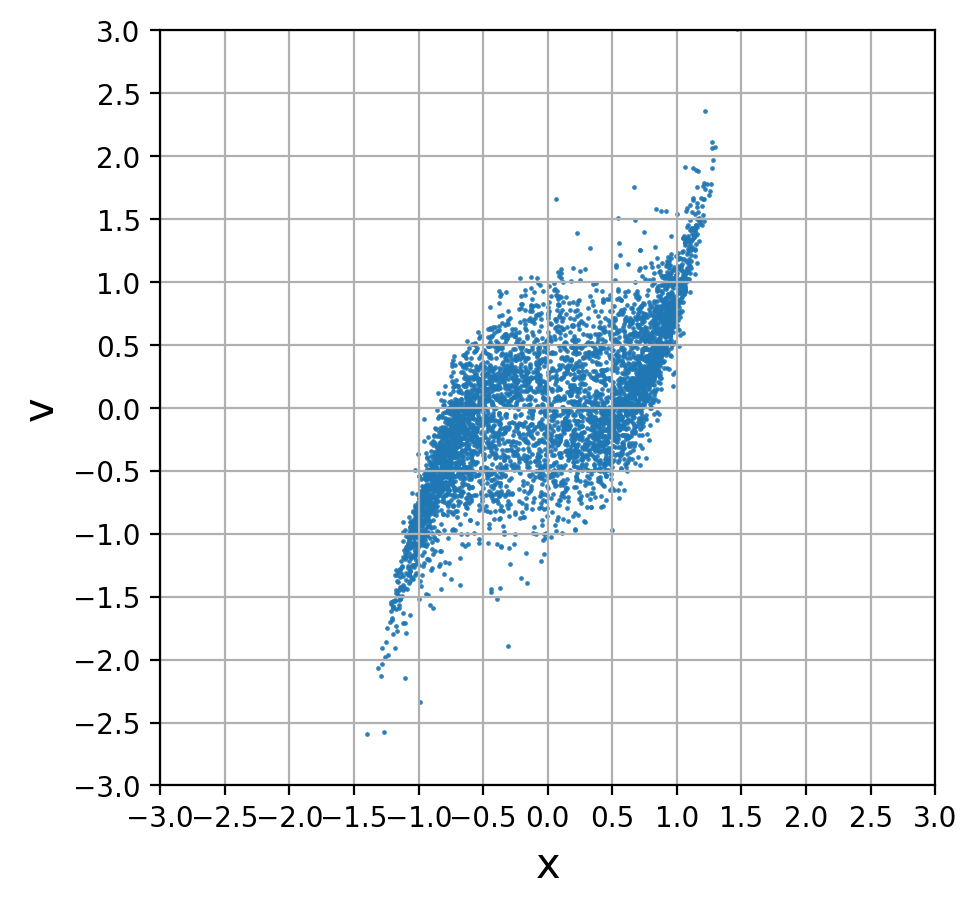}
	\end{subfigure}
	\begin{subfigure}[t]{0.32\textwidth}
		\centering
		\includegraphics[width=0.99\textwidth]{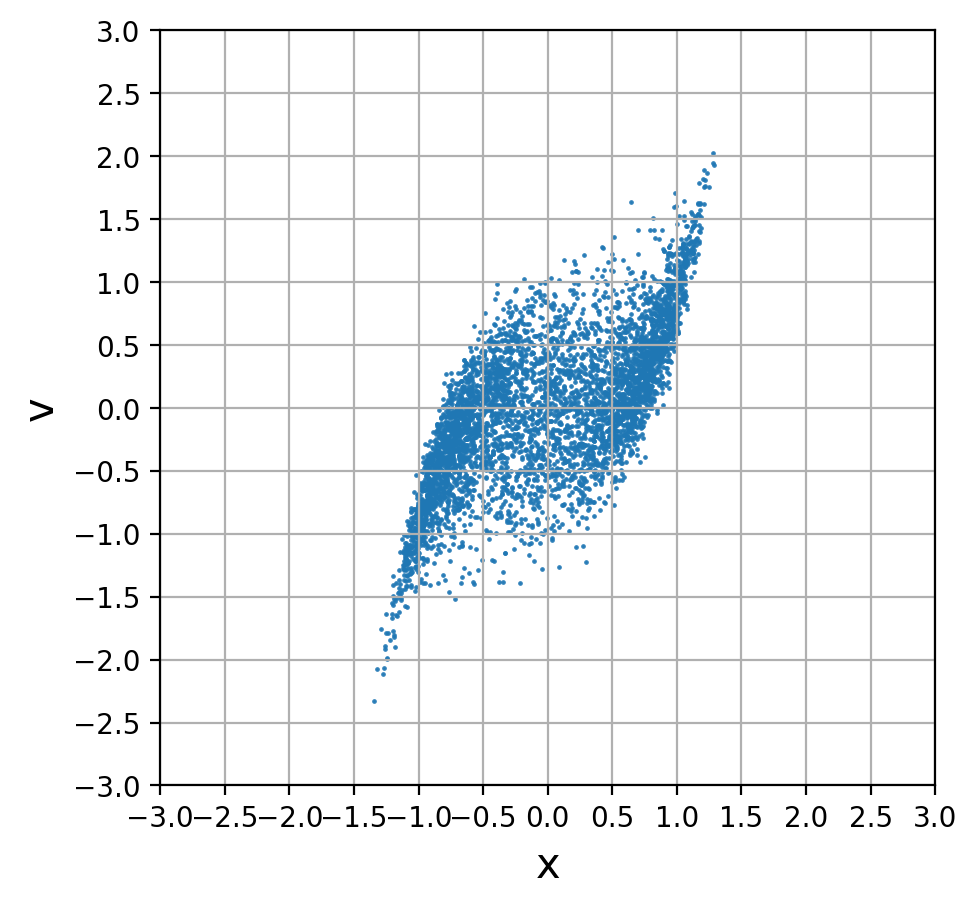}
	\end{subfigure}
	\begin{subfigure}[t]{0.32\textwidth}
		\centering
		\includegraphics[width=0.99\textwidth]{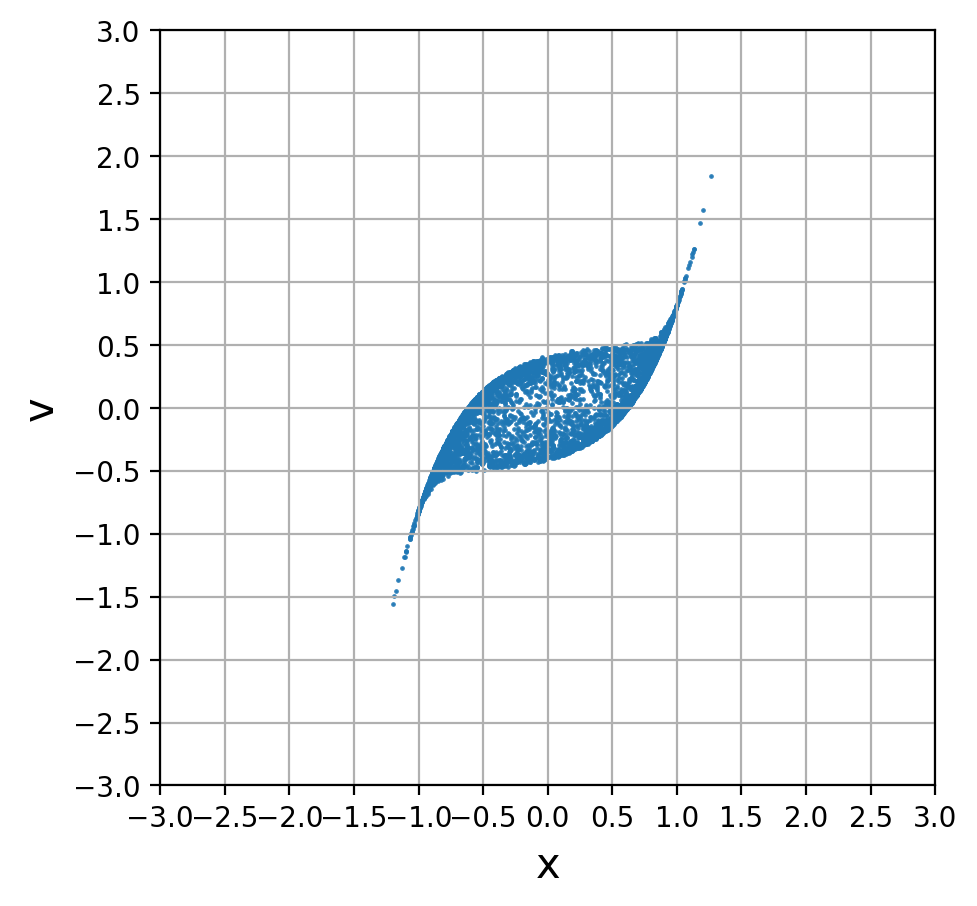}
	\end{subfigure}	
	\caption{5000 samples distribution in $xv$ plane. 
		Left: MSBTM, middle: SDE, right: SDE without noise. 
		From top to bottom: we present the samples at 0, 1000, 3000, 5000-th time steps
        respectively.}	
	\label{fig:Samples}
\end{figure}


\FloatBarrier
\section{Conclusions}
\label{sec:conclusions}

This paper studies the application of the score-based transport modeling to solving the mean-field Fokker-Planck equation, which extends the work \citep{Boffi} to a more general problem and saves the computational complexity for the interacting system with a large number of particles.
We prove an upper bound on the time derivative of the KL divergence to the estimated density from the solution of the equation as shown in Proposition~\ref{prop:msbtmseq} and provide an error analysis for the algorithm in Proposition~\ref{prop:error}. 
In numerical examples,
we consider the dynamics of interacting particle system from the mean field perspective so that we can avoid concatenating particles together and magnifying the dimension by the number of particles, and save the step of imposing particle permutation symmetry. 
Consequently, this extension simplifies equation solving and saves the computational cost when the number of particles is large.
Both quantitative and qualitative comparisons indicate MSBTM can approximate the solution as well as the SDE integration. 
In addition, MSBTM can compute quantities such as pointwise values of the density, entropy, that are not directly accessible when integrating SDE,
though MSBTM needs to learn the score and is more computationally expensive than integrating SDE.

\section*{Acknowledgement}
This work was done during YW's internship under the supervision of Professor Lu, in the Department of Mathematics, Duke University and Rhodes Information Initiative at Duke.
YX was partially supported by the Project of Hetao Shenzhen-HKUST Innovation Cooperation Zone HZQB-KCZYB-2020083.

\bibliographystyle{elsarticle-harv} 
\bibliography{main}


\appendix

\section{Proofs}
\label{app:proof}

Following the notation and assumptions in section~\ref{sec:assumptions}, first we prove a Lemma~\ref{lemma:entropy}, then Proposition~\ref{prop:msbtmseq} and
Proposition~\ref{prop:error}.

\subsection{Proof: Lemma~\ref{lemma:entropy}}
In Lemma~\ref{lemma:entropy}, it is shown that the time derivative of KL divergence to $\rho_t$ (solution of the transport equation~\eqref{eq:transport}) from  $\rho_t^{\star}$ (solution of the mean-field Fokker-Planck equation~\eqref{eq:FPE}) can be upper bounded by the difference between the neural network approximation $s_t$ and the score $\nabla \log \rho_t $.
\begin{restatable}[Control of the KL divergence]{lemma}{msbtment}
\label{lemma:entropy}
Let $\rho_t$ denote the solution to the transport equation~\eqref{eq:transport} 
and $\rho^{\star}_t$ denote the solution to the mean-field Fokker-Planck equation~\eqref{eq:FPE}. 
Under the assumptions in section~\ref{sec:assumptions}, we have 
\begin{equation}
		\begin{aligned}
			\frac{d}{dt} D_{\mathrm{KL}}(\rho_t \Vert \rho^{\star}_t)
			\le
			\frac12 \max\left\{ \hat{D}_t, \hat{C}_{K,t}
			\right\}
			\int_\Omega \|s_t(x)-\nabla \log \rho_t(x) \|^2 \rho_t(x) dx  ,
		\end{aligned}
	\end{equation}
where
$\hat{D}_t  := \sup_{x \in \Omega} \Vert D_t(x) \Vert < \infty, \Vert D_t(x) \Vert  = \sup_{z: \Vert z \Vert =1} z^\T D_t(x) z ; $ 
and 
\begin{equation*}
 	C_t := \sup_{x \in \Omega} 
 	\dfrac{	\lVert
 		\nabla \log \rho_t(x)-
 		\rho^{\star}_t(x)/\rho_t(x)  \nabla \log \rho_t^{\star}(x)
 		\lVert}{\lVert  \nabla \log \rho_t(x)-  \nabla \log  \rho^{\star}_t(x) \lVert} < \infty,
\end{equation*}
$\hat{C}_{K,t} :=\sup_{x,y \in \Omega } \vert \tilde{K}(x,y) \vert (1+ C_t^2) < \infty$ with $ \tilde{K} (x,y) = \int_{\Omega} K(x,y)dy.$
Furthermore, assuming that $\rho_{t=0}(x) = \rho^{\star}_{t=0}(x) $ for all $x \in \Omega$,
	we have for any $T\in [0,\infty)$,
	\begin{equation}
		\label{eq:ent}
		D_{\mathrm{KL}}(\rho_t \Vert \rho^{\star}_t) \le \frac12 \int_0^T 
		\max \left\{ \hat{D}_t, \hat{C}_{K,t}\right\} 
		\int_\Omega \|s_t(x) - \nabla \log \rho_t(x) \|^2 \rho_t(x)dx dt.
	\end{equation}
\end{restatable}

\begin{proof}
By definition, for $x \in \Omega \subseteq \mathbb{R}^{d},$
	$\rho_t (x)$ solves 
	$ \frac{\partial}{\partial t} \rho_t(x) = - \nabla \cdot (v_t(x, \rho_t) \rho_t(x)) $
	with  $$v_t(x, \rho_t) = b_t(x, \rho_t) - D_t(x) s_t(x)
	=f_t(x) - \int_{\Omega} K(x, y) \rho_t(y) dy- D_t(x) s_t(x) ,$$ 
	and $\rho_t^{\star} (x)$ solves 
	$	\frac{\partial}{\partial t} \rho_t^{\star}(x) = -\nabla \cdot (v^{\star}_t(x, \rho_t^{\star}) \rho_t^{\star}(x)) , $
	with $$ v^{\star}_t(x, \rho_t^{\star}) = b_t(x, \rho_t^{\star}) - D_t(x) \nabla \log\rho_t^{\star}
	=f_t(x) - \int_{\Omega} K(x, y) \rho_t^{\star}(y) dy - D_t(x) \nabla \log \rho_t^{\star}(x) .$$ 
 Denote $s^{\star}_t(x)=\nabla \log \rho_t^{\star} (x)$, 
then we have
	\begin{equation}
		\label{eq:ent:s1}
		\begin{aligned}
			&
			\frac{d}{dt} D_{\mathrm{KL}}(\rho_t \Vert \rho^{\star}_t) = \frac{d}{dt} \int_\Omega \log\left(\frac{\rho_t(x)}{\rho_t^{\star}(x)}\right) \rho_t(x) dx   \\
			& = -\int_\Omega \frac{\rho_t(x)}{\rho_t^{\star}(x)} \frac{\partial}{\partial t} \rho^{\star}_t(x) dx + \int_\Omega \log \left(\frac{\rho_t(x)}{\rho_t^{\star}(x)}\right) \frac{\partial}{\partial t} \rho_t(x)dx  \\
			& = -\int_\Omega v_t^{\star}(x, \rho_t^{\star}) \cdot \nabla \left(\frac{\rho_t(x)}{\rho_t^{\star}(x)}\right)  \rho^{\star}_t(x) dx + \int_\Omega v_t(x, \rho_t) \cdot \nabla \log \left(\frac{\rho_t(x)}{\rho_t^{\star}(x)}\right) \rho_t(x) dx   \\
			& = -\int_\Omega \left(v_t^{\star}(x, \rho_t^{\star}) - v_t(x, \rho_t)\right) \cdot \left(\nabla \log \rho_t(x)- \nabla \log \rho^{\star}_t(x)\right) \rho_t(x) dx    \\
			& = \int_\Omega \left(s_t^{\star}(x) - s_t(x)\right) \cdot D_t(x) \left(\nabla \log \rho_t(x)- s^{\star}_t(x)\right) \rho_t(x) dx \\
			&+ \int_\Omega \left(
			\int_{\Omega} K(x,y) \rho_t^{\star}(y) dy -\int_{\Omega} K(x,y) \rho_t(y) dy
			\right) \cdot \left(\nabla \log \rho_t(x)- s^{\star}_t(x)\right) \rho_t(x) dx.
		\end{aligned}
	\end{equation}
	Above, we used integration by parts to get the third equality. 
	Next, we have the bound on the mean field interaction term,
	\begin{equation}
		\label{eq:K}
		\begin{aligned}
			& \int_\Omega \left(
			\int_{\Omega} K(x,y) (\rho_t^{\star}(y)-\rho_t(y)) dy 
			\right) \cdot \left(\nabla \log \rho_t(x)- s^{\star}_t(x)\right) \rho_t(x) dx \\
			=&  
			\int_\Omega \left(
			\int_{\Omega} \tilde{K}(x,y)( \rho_t^{\star}(y) \nabla \log \rho_t^{\star}(y) -\rho_t(y) \nabla \log \rho_t(y)) dy 
			\right) 
			\cdot \left(\nabla \log \rho_t(x)- s^{\star}_t(x)\right) \rho_t(x) dx \\
			\le &
			\sup_{x,y \in \Omega } \vert \tilde{K}(x,y) \vert \int_\Omega  \int_\Omega 
			\bigg\lvert ( \frac{\rho^{\star}_t(y)}{\rho_t(y)}\nabla \log \rho_t^{\star}(y)-
			\nabla \log \rho_t(y)) 
			\cdot
			\left(\nabla \log \rho_t(x)- s^{\star}_t(x)\right) \bigg\lvert
			\rho_t(y) dy \rho_t(x)dx \\
			\le &
			\frac12
			\sup_{x,y \in \Omega } \vert \tilde{K}(x,y) \vert  \\
           &
			\left[ \int_\Omega 
			\lVert 
			\frac{\rho^{\star}_t(y)}{\rho_t(y)}\nabla \log \rho_t^{\star}(y)-
			\nabla \log \rho_t(y)
			\lVert^2
			\rho_t(y) dy 
			+ \int_\Omega  \lVert
			\nabla \log \rho_t(x)- s^{\star}_t(x)
			\lVert ^2 \rho_t(x)dx \right]		 \\
			\le &
			\frac12
			\sup_{x,y \in \Omega } \vert \tilde{K}(x,y) \vert
			(1+ C_t^2)
			\int_\Omega  \lVert
			\nabla \log \rho_t(x)- s^{\star}_t(x)
			\lVert ^2 \rho_t(x)dx.
		\end{aligned}
	\end{equation}
	The first equality in equation~\eqref{eq:K} is obtained using integration by parts with
	$ \tilde{K} (x,y) = \int_{\Omega} K(x,y) dy, x ,y\in \Omega.$
 For the second inequality in equation~\eqref{eq:K}, we use
	\begin{equation*}
		\begin{aligned}
			&
			\bigg \lvert ( \frac{\rho^{\star}_t(y)}{\rho_t(y)}\nabla \log \rho_t^{\star}(y)-
			\nabla \log \rho_t(y)) 
			\cdot
			\left(\nabla \log \rho_t(x)- s^{\star}_t(x)\right) \bigg \lvert \\
			\le &
			\frac12
			\lVert 
			\frac{\rho^{\star}_t(y)}{\rho_t(y)}\nabla \log \rho_t^{\star}(y)-
			\nabla \log \rho_t(y)
			\lVert^2
			+ \frac12
			\lVert
			\nabla \log \rho_t(x)- s^{\star}_t(x)
			\lVert ^2 , 
		\end{aligned}
	\end{equation*}
	and then
	\begin{equation*}
		\begin{aligned}
			&
			\int_\Omega  \int_\Omega
			\lVert 
			\frac{\rho^{\star}_t(y)}{\rho_t(y)}\nabla \log \rho_t^{\star}(y)-
			\nabla \log \rho_t(y)
			\lVert^2
			+  \lVert
			\nabla \log \rho_t(x)- s^{\star}_t(x)
			\lVert ^2 \rho_t(y) dy \rho_t(x)dx \\
			=&
			\int_\Omega \lVert 
			\frac{\rho^{\star}_t(y)}{\rho_t(y)}\nabla \log \rho_t^{\star}(y)-
			\nabla \log \rho_t(y)
			\lVert^2 \rho_t(y) dy \int_\Omega \rho_t(x)dx \\
		+ & 
			\int_\Omega \lVert
			\nabla \log \rho_t(x)- s^{\star}_t(x)
			\lVert ^2 \rho_t(x)dx \int_\Omega \rho_t(y) dy \\
			=&
			\int_\Omega \lVert 
			\frac{\rho^{\star}_t(y)}{\rho_t(y)}\nabla \log \rho_t^{\star}(y)-
			\nabla \log \rho_t(y)
			\lVert^2 \rho_t(y) dy
			+
			\int_\Omega \lVert
			\nabla \log \rho_t(x)- s^{\star}_t(x)
			\lVert ^2 \rho_t(x)dx .
		\end{aligned}
	\end{equation*}
	For the third(last) inequality in~\eqref{eq:K}, by definition and the continuity of $\rho^{\star}_t, \nabla \rho^{\star}_t, \rho_t, \nabla \rho_t $ over the domain $\Omega,$ it can be shown that
\begin{equation*}
		C_t :
		= \sup_{y \in \Omega} \dfrac{	\lVert
			\frac{\rho^{\star}_t(y)}{\rho_t(y)}\nabla \log \rho_t^{\star}(y)-
			\nabla \log \rho_t(y)
			\lVert}{\lVert \nabla \log \rho_t(y)- \nabla \log  \rho^{\star}_t(y) \lVert} < \infty,
	\end{equation*}
	then for any $y \in \Omega,$
	we have 
	$$ \dfrac{	\lVert
		\frac{\rho^{\star}_t(y)}{\rho_t(y)}\nabla \log \rho_t^{\star}(y)-
		\nabla \log \rho_t(y)
		\lVert}{\lVert \nabla \log \rho_t(y)- \nabla \log  \rho^{\star}_t(y) \lVert}
	\le C_t,$$
	and
	$$	\lVert
	\frac{\rho^{\star}_t(y)}{\rho_t(y)}\nabla \log \rho_t^{\star}(y)-
	\nabla \log \rho_t(y)
	\lVert
	\le
	C_t \lVert \nabla \log \rho_t(y)- \nabla \log  \rho^{\star}_t(y) \lVert .$$
	Adding \eqref{eq:ent:s1} and \eqref{eq:K}, 
	\begin{equation}
		\begin{aligned}
			&\frac{d}{dt} D_{\mathrm{KL}}(\rho_t \Vert \rho^{\star}_t) = 
			\int_\Omega \left(s_t^{\star}(x) - s_t(x)\right) \cdot D_t(x) \left(\nabla \log \rho_t(x)- s^{\star}_t(x)\right) \rho_t(x) dx \\
			&+ \int_\Omega \left(
			\int_{\Omega} K(x,y) \rho_t^{\star}(y) dy -\int_{\Omega} K(x,y) \rho_t(y) dy
			\right) \cdot \left(\nabla \log \rho_t(x)- s^{\star}_t(x)\right) \rho_t(x) dx   \\
			&\le
			\hat{D}_t
			\int_\Omega \left(s_t^{\star}(x) - s_t(x)\right) \cdot  \left(\nabla \log \rho_t(x)- s^{\star}_t(x)\right) \rho_t(x) dx \\
			&+
			\frac12
			\sup_{x,y \in \Omega } \vert \tilde{K}(x,y) \vert
			(1+ C_t^2)
			\int_\Omega  \lVert
			\nabla \log \rho_t(x)- s^{\star}_t(x)
			\lVert ^2 \rho_t(x)dx \\
			& \le
			\frac12 \max\left\{ \hat{D}_t, \hat{C}_{K,t}
			\right\}
			\int_\Omega \|s_t(x)-\nabla \log \rho_t(x) \|^2 \rho_t(x) dx  ,
		\end{aligned}
	\end{equation}
	with  $\hat{D}_t := \sup_{x \in \Omega} \Vert D_t(x) \Vert < \infty, \Vert D_t(x) \Vert  = \sup_{z: \Vert z \Vert =1} z^\T D_t(x) z , $
	and  $\hat{C}_{K,t} :=\sup_{x,y \in \Omega } \vert \tilde{K}(x,y) \vert (1+ C_t^2).$
	The last inequality is obtained using
	\begin{equation}
		\begin{aligned}
			\| s_t - \nabla \log \rho_t \|^2 
			& = \|s_t -s^{\star}_t+s^{\star}_t - \nabla \log \rho_t \|^2 \\
			& = \| s_t - s^{\star}_t \|^2
			+\|\nabla \log \rho_t  - s^{\star}_t  \|^2
			+ 2 (\nabla \log \rho_t -s^{\star}_t)\cdot (s^{\star}_t - s_t) \\
			& \ge \|\nabla \log \rho_t -s^{\star}_t \|^2 + 2 (\nabla \log \rho_t -s^{\star}_t)\cdot (s^{\star}_t - s_t) .
		\end{aligned}
	\end{equation}
	Finally
	\begin{equation}
		\begin{aligned}
			\frac{d}{dt} D_{\mathrm{KL}}(\rho_t \Vert \rho^{\star}_t)
			\le
			\frac12 \max\left\{ \hat{D}_t, \hat{C}_{K,t}
			\right\}
			\int_\Omega \|s_t(x)-\nabla \log \rho_t(x) \|^2 \rho_t(x) dx    .
		\end{aligned}
	\end{equation}
	Integrating this equation on $t\in[0,T],$ 
	we have for any $T\in [0,\infty)$,
	\begin{equation}
		\begin{aligned}
		    D_{\mathrm{KL}}(\rho_t \Vert \rho^{\star}_t) 
      & \le 
		D_{\mathrm{KL}}(\rho_{t=0} \Vert \rho^{\star}_{t=0})  \\
      & \quad +
		\frac12 \int_0^T 
		\max \left\{ \hat{D}_t, \hat{C}_{K,t}\right\} 
		\int_\Omega \|s_t(x) - \nabla \log \rho_t(x) \|^2 \rho_t(x)dx dt.
		\end{aligned}
	\end{equation}
	With $\rho_{t=0}(x) = \rho^{\star}_{t=0}(x) $ for all $x \in \Omega$,
	\begin{equation}
		D_{\mathrm{KL}}(\rho_T \Vert \rho^{\star}_T) \le 
		\frac12 \int_0^T 
		\max \left\{ \hat{D}_t, \hat{C}_{K,t}\right\} 
		\int_\Omega \|s_t(x) - \nabla \log \rho_t(x) \|^2 \rho_t(x)dx dt.
	\end{equation}
\hfill
\end{proof}


\subsection{Proof: Proposition~\ref{prop:msbtmseq}}
\label{app:proof p1}

\begin{proof}
Given a function $ \phi(x)$ of interest,
we have  
\begin{equation}
    \int_{\Omega} \phi(x) \rho_t(x) dx =\int_{\Omega} \phi(X_{0,t}(x)) \rho_0(x) dx ,
\end{equation}
using the property of the transport map $X_{0,t}(\cdot)$.
Thus,
	\begin{equation}
		\begin{aligned}
			&\int_\Omega \left( \|s_t(X_{0,t}(x))\|^2 +2 \nabla \cdot s_t(X_{0,t}(x)) \right) \rho_0(x) dx \\
			=& \int_\Omega \left( \|s_t(x) \|^2 +2 \nabla \cdot s_t(x) \right) \rho_t(x) dx \\
			=& \int_\Omega \left( \|s_t(x) \|^2 -2  s_t(x)\cdot \nabla \log \rho_t (x)\right) \rho_t(x) dx \\
               =& \int_\Omega \left( \|s_t(X_{0,t}(x)) \|^2 -2  s_t(X_{0,t}(x))\cdot \nabla \log \rho_t (X_{0,t}(x))\right) \rho_0(x) dx.
		\end{aligned}
	\end{equation}
	Then the minimization~\eqref{eq:sbtm3} is equivalent to 
	the following minimization problem
	\begin{equation}
		\min_{s_t}  
		\int_\Omega \|s_t(X_{0,t}(x)) - \nabla \log \rho_t(X_{0,t}(x))\|^2 \rho_0(x) dx,
	\end{equation}
	since  
	\begin{equation}
		\begin{aligned}
			&\int_\Omega  \|s_t(X_{0,t}(x))- \nabla \log \rho_t (X_{0,t}(x))\|^2 \rho_0(x) dx \\
			=& \int_\Omega \left( \|s_t(X_{0,t}(x)) \|^2 -2  s_t(X_{0,t}(x))\cdot \nabla \log \rho_t (X_{0,t}(x)) + \|\nabla \log \rho_t(X_{0,t}(x)) \|^2 \right) \rho_0(x) dx,
		\end{aligned}
	\end{equation}
	and $\|\nabla \log \rho_t(X_{0,t}(x)) \|^2$ can be neglected with respect to $s_t(x)$.
	The minimizer $\hat{s}_t$ of the minimization problem~\eqref{eq:sbtm3} should satisfy $\hat{s}_t(x) = \nabla \log \rho_t (x).$
	Then the equation~\eqref{eq:transport} becomes
	\begin{equation}
		\frac{\partial}{\partial t} \rho_t(x) = -\nabla \cdot \left(b_t(x, \rho_t) \rho_t(x) -D_t(x) \nabla \log \rho_t(x) \right),
	\end{equation}
	which recovers the mean-field Fokker-Planck equation~\eqref{eq:FPE}.
Second, we can obtain the inequality~\eqref{eq:entD} from Lemma~\ref{lemma:entropy}.
\hfill
\end{proof}

\subsection{Proof: Proposition~\ref{prop:error}}
\label{app:proof p2}

\begin{proof}
For convenience, we use $X^{\star}_{t}, X^{N}_{t} $ as $X^{\star}_{t_0,t}, X^{N}_{t_0,t}.$
For $t_0 \leq t \leq t_0 + (N_T -1)\Delta t,$
using Taylor expansion and the smoothness of the velocity $v^{\star},$
\begin{equation}\label{eq:ea-X}
    \begin{aligned}
        & X^{\star}_{t+\Delta t}(x) 
        =X^{\star}_{t}(x) + v_t^{\star} (X^{\star}_t(x), \rho^{\star}(X^{\star}_t(x))) \Delta t  + \frac{d^2 X_{t}^{\star}(x) }{dt^2}\bigg\vert_{t=\tau}(\Delta t )^2\\
        =& X^{\star}_{t}(x)  + \frac{d^2 X_{t}^{\star}(x) }{dt^2}\bigg\vert_{t=\tau}(\Delta t )^2  \\
         +&
        \underbrace{\left[f_t (X^{\star}_{t}(x)) - \int_{\Omega} K(X^{\star}_{t}(x),y) \rho^{\star}_t(y) dy - D_t (X^{\star}_{t}(x)) \nabla \log \rho_t^{\star} (X^{\star}_{t}(x))\right]}_{v^{\star}_t \left(X^{\star}_t(x)  \right)} \Delta t ,
    \end{aligned}
\end{equation}
with $\tau \in \left[t, t+\Delta t\right].$
According to Algorithm~\ref{alg:msbtm},
\begin{equation}\label{eq:ea-Xn}
    \begin{aligned}
       & X^{N}_{t+ \Delta t} (x) = X^{N}_{t}(x) \\
   & +
   \underbrace{\left[  f_t(X^{N }_{t}(x) )
         - \frac1N \sum_{j=1}^N K(X^{N }_{t}(x), X^{N }_{t} (x^{(j)}) )  
        -D_{t }(X^{N }_{t}(x)) s_{t}(X^{N }_{t}(x)) \right] }_{v^{N}_t \left(X^{N}_t(x) \right)}\Delta t ,
    \end{aligned}
\end{equation}
where $\{x^{(i)}\}_{i=1}^N$ is a set of $N$ samples from initial density $\rho_{t_0}$.
Subtracting equation~\ref{eq:ea-X} by equation~\ref{eq:ea-Xn}, we have
\begin{equation}\label{eq:gap}
   \begin{aligned}
 X^{\star}_{t+\Delta t}(x) - & X^{N}_{t+ \Delta t} (x) =X^{\star}_{t }(x) -X^{N}_{t } (x)  
     + \frac{d^2 X_{t}^{\star}(x) }{dt^2}\bigg\vert_{t=\tau}(\Delta t )^2 \\
    + & \nabla f (X_{t,\xi})\cdot (X^{\star}_{t }(x) -X^{N}_{t } (x) )
     \Delta t \\
     -& \left[\int_{\Omega} K(X^{\star}_{t}(x),y) \rho^{\star}_t(y) dy 
     -\frac1N \sum_{j=1}^N K(X^{N }_{t}(x), X^{N }_{t} (x^{(j)}) ) 
    \right] \Delta t \\
    -& \left[D_t (X^{\star}_{t}(x)) \nabla \log \rho_t^{\star} (X^{\star}_{t}(x))
    -D_{t} (X^{N }_{t}(x)) s_{t}(X^{N }_{t}(x))
    \right] \Delta t .
   \end{aligned}
\end{equation}
First, based on the law of large numbers and assumption that the kernel function is twice-differentiable,
\begin{equation}
   \begin{aligned}
       &\mathbb{E}_{x, x^{(j)}} \bigg\Vert \int_{\Omega} K(X^{\star}_{t}(x),y) \rho^{\star}_t(y) dy 
     -\frac1N \sum_{j=1}^N K(X^{N }_{t}(x), X^{N }_{t} (x^{(j)})) \bigg \Vert \\
    =& \mathbb{E}_{x, x^{(j)}} \bigg\Vert  \int_{\Omega} K(X^{\star}_{t}(x),y) \rho^{\star}_t(y) dy 
    - \frac1N \sum_{j=1}^N K(X^{\star }_{t}(x), X^{\star }_{t} (x^{(j)}) ) \\
    +&\frac1N \sum_{j=1}^N K(X^{\star }_{t}(x), X^{\star }_{t} (x^{(j)}) ) -\frac1N \sum_{j=1}^N K(X^{N }_{t}(x), X^{N }_{t} (x^{(j)}) ) \bigg \Vert \\
    \leq &
    \mathbb{E}_{x, x^{(j)}}  \bigg\Vert  \int_{\Omega} K(X^{\star}_{t}(x),y) \rho^{\star}_t(y) dy 
    - \frac1N \sum_{j=1}^N K(X^{\star }_{t}(x), X^{\star }_{t} (x^{(j)}) ) \bigg\Vert  \\
    +& \mathbb{E}_{x, x^{(j)} } \frac1N \sum_{j=1}^N   \bigg\Vert 
    K(X^{\star }_{t}(x), X^{\star }_{t} (x^{(j)}) )
    - K(X^{N }_{t}(x), X^{N }_{t} (x^{(j)}) ) \bigg \Vert  \\
    \leq & C_1 \frac{1}{\sqrt{N}} + C_2 \mathbb{E}_{x} \Vert X^{\star}_t(x) -X_{t}^{N} (x)\Vert,
   \end{aligned}
\end{equation}
where $C_1, C_2>0$ are some constants and $x, x^{(j)}, j=1, \dots, N$ are independent and identically distributed.
Second
\begin{equation}
    \begin{aligned}
    & \Vert D_t (X^{\star}_{t}(x)) \nabla \log \rho_t^{\star} (X^{\star}_{t}(x))
    -D_{t }(X^{N }_{t}(x)) s_{t}(X^{N }_{t}(x)) \Vert \\
    =&
   \Vert D_t (X^{\star}_{t}(x)) \nabla \log \rho_t^{\star} (X^{\star}_{t}(x))
    - D_t (X^{N }_{t}(x)) \nabla \log \rho_t^{\star} (X^{N }_{t}(x)) \\
    +& 
     D_t (X^{N }_{t}(x)) \nabla \log \rho_t^{\star} (X^{N}_{t}(x))
    -D_{t }(X^{N }_{t}(x)) s_{t}(X^{N }_{t}(x))\Vert \\
    \leq &
    C_3 \Vert X^{\star}_t(x) -X_{t}^{N} (x)\Vert
    + C_4 \epsilon,
    \end{aligned}
\end{equation}
where $C_3, C_4>0$ are some constants.
Denote 
$\mathcal{E}_{t} := \mathbb{E}_{x }\left(\Vert X^{\star}_t(x) -X_{t}^{N} (x)\Vert  \right),$ 
then from equation~\eqref{eq:gap}, we have, for some constant $C>0,$
\begin{equation}
    \mathcal{E}_{t+\Delta t} \leq \mathcal{E}_{t}  +  
    C \mathcal{E}_{t}   \Delta t
    +\left[ C_1 \frac{1}{\sqrt{N}} 
    + C_4 \epsilon \right]\Delta t
     + \mathcal{O}((\Delta t )^2) .
\end{equation} 
Therefore,
\begin{equation}
     \begin{aligned}
          \mathcal{E}_{t} 
          & \le e^{C(t-t_0)}\mathcal{E}_{t_0}  + \mathcal{O}(\frac{1}{\sqrt{N}}) (t-t_0)+ \mathcal{O}(\epsilon) (t-t_0) +\mathcal{O}(\Delta t )(t-t_0),
     \end{aligned}
\end{equation}
as $\frac1N, \epsilon, \Delta t \rightarrow 0$.
With the same initial condition that $X_{t_0}^{\star}(x)= X_{t_0}^{N}(x)$ so $\mathcal{E}_{t_0}=0,$
and we obtain equation~\eqref{eq:error-X}.
\end{proof} 

\section{Experimental details}
All numerical experiments were performed in jax, using the dm-haiku package to
implement the networks and the optax package for optimization.

\subsection{Quantitative comparison in section~\ref{sec:ex1}}
\label{app:qc}

We provide the calculation and derivation of the quantities used in the experiments.
In the following, we denote $X_{t_0,t}(x^{(i)})$ as $X_t^{(i)}$ for short.

\paragraph{Mean and covariance}
The empirical mean $\hat{m}_t$ and covariance $\hat{C}_t$  can be directly calculated as
\begin{align}
	\hat{m}_t &= \frac{1}{N}\sum_{i=1}^{N} X_t^{(i)}, \\
	 \hat{C}_t &= 
	\frac{1}{N}\sum_{i=1}^{N}  (X_t^{(i)}- \hat{m}_t)(X_t^{(i)}- \hat{m}_t)^T,
\end{align}
using samples $ \left\{ X_t^{(i)}\right\}_{i=1}^N$.
The analytical mean $m_t$ and covariance $C_t$ 
can be obtained by solving the ODE
\begin{equation}\label{eq:MC}
	\begin{aligned}
		dm_t &= (\beta_t -m_t)dt, \\
		dC_t & = (-2C_t-2\alpha (1-\frac{1}{N}) C_t + 2DI_2)dt .
	\end{aligned}
\end{equation}
The ODE of mean $m_t$ is derived from taking expectation of the dynamics~\eqref{eq:harmonic_SDE}.
And for the covariance $C_t:$
starting from the dynamics
\begin{equation}
	dX^{(i)} _{t}=  (\beta_t - X_t^{(i)}) dt -\alpha \frac{1}{N} \sum_{j=1}^{N} (X^{(i)} _{t} -  X^{(i)} _{t} ) dt 
	+\sqrt{2 D} dW_t, i=1, \dots N,
\end{equation}
we have
\begin{align*}
	& d\left(X^{(i)} _{t} (X^{(i)} _{t})^\T \right) =
	dX^{(i)} _{t}  (X^{(i)} _{t} )^\T +X^{(j)} _{t} (dX^{(i)} _{t} )^\T+2DI_2 dt      \\
	=& \left[\beta_t - X_t^{(i)}- 
	\alpha \frac{1}{N} \sum_{j=1}^{N} (X^{(i)} _{t} -  X^{(i)} _{t} )\right] (X^{(i)} _{t} )^\T dt \\
	+&  X^{(i)} _{t} \left[\beta_t - X_t^{(i)}- 
	\alpha \frac{1}{N} \sum_{j=1}^{N} (X^{(i)} _{t} -  X^{(i)} _{t} )\right]^\T dt + 2DI_2 dt + (2\sqrt{2 D} X^{(i)} _{t}  )dW_t.
\end{align*}
Taking the expectation,
\begin{align*}
	&	d \mathbb{E} \left(X^{(i)} _{t} (X^{(i)} _{t})^\T \right) -2DI_2 dt
	=\mathbb{E}
	\left\{	\left[\beta_t - X_t^{(i)}- 
	\alpha \frac{1}{N} \sum_{j=1}^{N} (X^{(i)} _{t} -  X^{(i)} _{t} )\right] (X^{(i)} _{t} )^\T\right\} dt \\
	& \quad \quad +\mathbb{E} \left\{X^{(i)} _{t} \left[\beta_t - X_t^{(i)}- 
	\alpha \frac{1}{N} \sum_{j=1}^{N} (X^{(i)} _{t} -  X^{(i)} _{t} )\right]^\T \right\} dt  \\
	&= \left\{\beta_t m_t^\T + m_t \beta_t^\T -2\mathbb{E}(X^{(i)} _{t} (X^{(i)} _{t})^\T)  -  	2\alpha \frac{1}{N} \sum_{j=1}^{N}
	\left[\mathbb{E}(X^{(i)} _{t} (X^{(i)} _{t})^\T)  -\mathbb{E}(X^{(i)} _{t} (X^{(j)} _{t})^\T) \right] \right\} dt  \\
	&= \left\{\beta_t m_t^\T + m_t \beta_t^\T -2\mathbb{E}(X^{(i)} _{t} (X^{(i)} _{t})^\T)  -  	2\alpha \frac{1}{N} \sum_{j\neq i}
	\left[\mathbb{E}(X^{(i)} _{t} (X^{(i)} _{t})^\T)  -m_t m_t^\T) \right] \right\} dt  \\ 
	&=  \left\{\beta_t m_t^\T + m_t \beta_t^\T -2\mathbb{E}(X^{(i)} _{t} (X^{(i)} _{t})^\T)  -  	2\alpha \frac{1}{N} \sum_{j\neq i}
	\left[\mathbb{E}(X^{(i)} _{t} (X^{(i)} _{t})^\T)  -m_t m_t^\T) \right] \right\} dt .
\end{align*}
and using $$
m_t = \mathbb{E}  (X^{(i)} _{t}), \quad
C_t = Var(X_t) = \mathbb{E}((X_t-\mathbb{E}(X_t)) (X_t-\mathbb{E}(X_t))^\T) = \mathbb{E} \left(X^{(i)} _{t} (X^{(i)} _{t})^\T \right) -m_t m_t^\T,  $$
we have
\begin{align*}
	dC_t &= d \mathbb{E} \left(X^{(i)} _{t} (X^{(i)} _{t})^\T \right) -dm_t m_t^\T - m_t (dm_t)^\T   \\
	&=  d \mathbb{E} \left(X^{(i)} _{t} (X^{(i)} _{t})^\T \right)-(\beta_t -m_t)m_t^\T  dt - m_t(\beta_t -m_t)^\T  dt\\
	&=(-2 C_t - 2\alpha \frac{N-1}{N} C_t + 2D I_2)dt.
\end{align*}

\paragraph{Relative Fisher divergence}

The solution is Gaussian for all $t:$ 
denote the analytical solution
$\rho_t = N(m_t, C_t),$
then the corresponding score is $-\nabla \log \rho_t(x) = C_t^{-1} (x-m_t).$
We use the prediction of the network $s_t(\cdot)$ on training data or SDE data to calculate the relative Fisher divergence
\begin{equation}
	\frac{\int_\Omega |s_t(x) - \nabla\log\rho_t(x)|^2 \bar{\rho}(x) dx}{\int_\Omega|\nabla\log\rho_t(x)|^2 \bar{\rho}(x)dx}
	\approx 
	\frac{\sum_{i=1}^{N} |s_t(X^{(i)}_t) - \nabla\log\rho_t(X^{(i)}_t)|^2 }{\sum_{i=1}^{N} |\nabla\log\rho_t(X^{(i)}_t)|^2 },
\end{equation} 
where $\left\{X^{(i)}_t\right\}$ are training data or SDE data at $t,$ obtained from the algorithm or SDE integration.

\paragraph{Analytical entropy and entropy production rate} 
Entropy production rate $\frac{dE_t}{dt}$ can be
numerically computed over the training data,
\begin{align*}
	\frac{d}{dt}E_t
	&=- \int_{\Omega} \log\rho_t(x) \frac{\partial}{\partial t} \rho_t (x) dx , \quad
	\text{ with } \frac{\partial}{\partial t} \rho_t (x)  = -\nabla \cdot (v_t(x) \rho_t (x)),
	\\
	&=\int_{\Omega} \log\rho_t(x) \nabla \cdot (v_t(x) \rho_t (x)) dx
	= -\int_{\Omega} \nabla \log\rho_t(x) \cdot v_t(x) \rho_t (x) dx \\
	&=  -\int_{\Omega}s_t(x)\cdot v_t(x) \rho_t(x)dx 
	\approx -\frac{1}{N} \sum_{i=1}^{N} s_t(X_t^{(i)}) \cdot v_t(X_t^{(i)}),
\end{align*}
with 
\begin{align*}
	v_t(x) &=  \beta_t - x  -  \alpha \int_{\Omega} (x -y) \rho_t(y) dy - D \nabla \log \rho_t(x)  \\
	&\approx
	\beta_t - x
	- 	\frac{\alpha}{N} \sum_{j=1}^N 
	\left(x - X^{(j)}_t \right)
	- D s_t(x).
\end{align*}
Analytically,
the differential entropy $E_t$ is  
\begin{equation}
	E_t =-\int_{\mathbb{R}^{\bar{d}}} \log\rho_t(x) \rho_t(x) dx = \frac{1}{2} \bar d \left(\log\left(2\pi\right) + 1\right) + \frac{1}{2}\log\det C_t,
\end{equation} 
and then entropy production rate is
\begin{align*}
	\frac{d}{dt} E_t 
	&= \frac{1}{2}	\frac{d}{dt} \log\det C_t
	=\frac{1}{2}\Tr\left((\frac{d  \log\det C_t }{dC_t})^\T \frac{d}{dt} C_t\right) \\
	&=\frac{1}{2}\Tr\left( C_t^{-1}  \frac{d}{dt} C_t\right)
	=\frac{1}{2}\Tr\left( (-2-2\alpha \frac{N-1}{N})I_2 + 2D C_t^{-1}\right)	 \\
	&= -2-2\alpha \frac{N-1}{N} + D\Tr(C_t^{-1}) .
\end{align*}


\end{document}